\documentclass{article}[11pt]
\usepackage{amsmath}
\usepackage{amssymb}
\usepackage{amsthm}
\numberwithin{equation}{section}

\newcommand{\ud}{\,\mathrm{d}}

\newcommand{\eps}{\varepsilon}
\newcommand{\R}{{\mathbb{R}}}
\theoremstyle{plain}
\newtheorem{Prop} {Proposition} 
\newtheorem{lemma} {Lemma}
\newtheorem{theorem}{Theorem}

\theoremstyle{remark}

\author{ Pierre-Emmanuel Jabin,\footnote{CSCAMM and Dept. of Mathematics, University of Maryland,
College Park, MD 20742, USA. P.E. Jabin is partially supported by NSF Grant 1312142 and by NSF Grant RNMS (Ki-Net) 1107444.}
\quad Zhenfu Wang
\footnote{CSCAMM and Dept. of Mathematics, University of Maryland,
College Park, MD 20742, USA. Z. Wang is supported by NSF Grant 1312142.}
}
\title{Mean Field Limit and Propagation of Chaos for Vlasov Systems with Bounded Forces}
\date{}

\begin{document}
 \maketitle
\begin{abstract}
We consider large systems of particles interacting through rough but bounded interaction kernels. We are able to control the relative entropy between the $N$-particle distribution and the expected limit which solves the corresponding Vlasov system. This implies the Mean Field limit to the Vlasov system together with Propagation of Chaos through the strong convergence of all the marginals. The method works at the level of the Liouville equation and relies on precise combinatorics results. 
\end{abstract}
\tableofcontents
\section{Introduction} 
 Consider the classical Newton dynamics for $N$ indistinguishable point-particles. Denote by $X_i \in \Omega$ and $V_i \in \mathbb{R}^d$ the position and velocity of particle number $i$. The space domain $\Omega$ may be the whole space $\mathbb{R}^d$ or the periodic torus $\mathbb{T}^d$.  The evolution of the system is given by the following ODEs, (in a precise and weak sense defined below in subsection \ref{weaksolution})
\begin{equation}
\label{ODE}
\begin{cases}
\dot{X}_i =V_i,\\
\dot{V}_i= \frac{1}{N}\sum_{j \ne i} K(X_i- X_j),
\end{cases}
\end{equation}
where $i =1, \cdots, N$. We use the so-called 
mean-field scaling which consists in keeping the total
mass (or charge) of order $1$: this explains the $1/N$ factor in front
of the force terms. 

Our method applies in an identical manner to stochastic models, hence we will consider in general the system of stochastic differential equations
\begin{equation}
\label{SDE}
\begin{cases}
\ud  X_i=V_i  \ud t,\\
\ud {V}_i= \frac{1}{N}\sum_{j \ne i} K(X_i- X_j) \ud t+ \sqrt{2\eps_N}  \ud W_i^t,
\end{cases}
\end{equation} 
where the $W_i^t$ are $N$ independent Wiener processes (Brownian motions), which may model various type of random phenomena: For instance random collisions against a given background. The stochastic part is scaled with the parameter $\eps_N$. Our approach is completely independent of the choice of $\eps_N$ so we will handle at the same time 
\begin{itemize}
\item No randomness $\eps_N=0$ where \eqref{SDE} reduces to the deterministic \eqref{ODE}.
\item Fixed randomness $\eps_N\rightarrow \eps>0$ as $N\rightarrow +\infty$. 
\item Vanishing randomness $\eps_N\rightarrow \eps=0$ as $N\rightarrow +\infty$.
\end{itemize}
The best known example of interaction kernel is the Coulombian or gravitational force, $K(x)=C\,x/|x|^d$. In this article, we however consider {\em bounded interaction kernels with no additional regularity}, meaning that we only assume that $K\in L^\infty$. While this does not cover the Coulombian case, it significantly expands the interaction kernels for which one can prove the mean field limit and propagation of chaos, including for very oscillatory kernels.

As $N\rightarrow\infty$, one expects that the system of particles will converge to a continuous PDE model, the Vlasov or McKean-Vlasov (with diffusion) equation 
\begin{equation}
\label{PDE}
\partial_t f + v \cdot \nabla_x f + K \star \rho \cdot \nabla_v f-\eps\,\Delta_v f=0,
\end{equation}
where $f=f(t,x, v)$ is the phase space density while $\rho(t, x) =\int_{\mathbb{R}^d} f(t, x, v) \ud v$ is the macroscopic density. 

The main goal of this article is to derive the Vlasov Eq. \eqref{PDE} from the systems \eqref{ODE} or \eqref{SDE} and quantify the convergence. We give the precise notions of convergence (together with the definitions of Mean Field limit and propagation of chaos) in subsection \ref{mainresults}.

{\bf Notations}: We denote $X=(x_1, \cdots, x_N)$ and $V=(v_1, \cdots, v_N)$ while keeping $x\in \Omega$ and $v\in\R^d$ for the variables at the limit. We also use $z_i=(x_i, v_i)$, $z=(x,\;v)$ and $Z=(z_1, \cdots, z_N)$.
\subsection{Which solution for the ODE system: The Liouville Equation\label{weaksolution}} 
Even before considering the limit $N\rightarrow \infty$, the first nontrivial question is which notion of solution one can use for the system of ODEs \eqref{ODE} (and to a lesser degree for \eqref{SDE}). Indeed as $K$ is only bounded, we are quite far from the classical Cauchy-Lipschitz theory, requiring $K$ locally Lipschitz.

As usual for this type of question, we consider instead of \eqref{ODE}, the Liouville equation
\begin{equation}
\label{Liouville}
\partial_t f_N + \sum_{i=1}^N \Big( v_i \cdot \nabla_{x_i}f_N +\frac{1}{N} \sum_{j \ne i} K(x_i-x_j) \cdot \nabla_{v_i} f_N \Big)=\eps_N\sum_{i=1}^N \Delta_{v_i} f_N.
\end{equation}
Defining the Liouville operator as 
\begin{equation*}
L_N =\sum_{i=1}^Nv_i \cdot \nabla_{x_i} + \frac{1}{N}\sum_{i=1}^N \sum_{j \ne i} K(x_i-x_j) \cdot \nabla_{v_i}-\eps_N\sum_{i=1}^N \Delta_{v_i},
\end{equation*}
the Liouville equation can be written as 
\begin{equation*}
\partial_t f_N + L_N f_N=0.
\end{equation*}
One advantage of our approach is that we only need weak solutions to \eqref{Liouville}, {\em i.e.} solutions in the sense of distribution as per
\begin{Prop} {\bf Existence of weak solution of Liouville equation (\ref{Liouville}).}\\ 
\label{WeaSolLio}
Assume that $K\in L^\infty$ and that the initial data $f_N^0\geq 0$ satisfies the following assumptions
\begin{equation}
\label{AssLiouWeak1}\begin{split}
&i) \quad f_N^0\in L^1((\Omega\times \R^d)^N)\ \mbox{with}\ \int_{(\Omega\times \R^d)^N} f_N^0 \ud Z=1,\\
& ii) \quad \int_{(\Omega \times \mathbb{R}^d)^N} f_N^0 \log f_N^0 \ud Z < \infty,
\end{split}
 \end{equation}
together with the moment assumptions
\begin{equation}
\label{AssLiouWeak2}
iii) \quad  \int_{(\Omega \times \mathbb{R}^d)^N} \sum_{i=1}^N \left( 1+ |x_i|^{2k} + |v_i|^{2k}\right) f_N^0 \ud Z < \infty,
\end{equation}
for some $k>0$. 
Then there exists $f_N\geq 0$ in $L^\infty(\R_+,\;L^1(\Omega\times \R^d)^N)$ solution to \eqref{Liouville} in the sense of distribution and satisfying
\begin{equation}
\label{propLiouWeak}\begin{split}
&i)\quad \int_{(\Omega\times \R^d)^N} f_N(t,Z)\ud Z=1,\quad \mbox{for}\ a.e.\ t,\\
&ii)\quad \int_{(\Omega \times \mathbb{R}^d)^N} f_N(t,Z)\, \log f_N (t, Z)\ud Z +\eps_N \int_0^t \int_{(\Omega \times \mathbb{R}^d)^N} \frac{|\nabla_V f_N(s,Z)|^2}{f_N(s,Z )} \ud Z \ud s \\
&\qquad\qquad\qquad\leq \int_{(\Omega \times \mathbb{R}^d)^N} f_N^0 \log f_N^0 \ud Z, \quad \mbox{for}\ a.e.\ t,\\
&iii) \sup_{t\in [0,\ T]} \int_{(\Omega \times \mathbb{R}^d)^N} \sum_{i=1}^N \left( 1+ |x_i|^{2k} + |v_i|^{2k}\right)   f_N(t, Z) \ud Z < \infty, \ \mbox{for any} \ T<\infty.
\end{split}
\end{equation}
\end{Prop}
We omit the proof of Proposition \ref{WeaSolLio}. It is straightforward by approximating $K$ by a sequence of smooth kernels $K_\epsilon$ and then passing to limit.

It is important to emphasize here that we do not have uniqueness in Prop. \ref{WeaSolLio}: There could very well be several such solutions.  
Uniqueness and in general the well-posedness of the Cauchy problem for advection equations like \eqref{Liouville} are usually handled through the theory of renormalized solutions as introduced in \cite{DL} and improved in \cite{Am} (we refer to \cite{AmCr} and \cite{DeL} for a very good introduction to the theory). 

Renormalized solutions not only give well-posedness to advection equations like \eqref{Liouville} but also provide the existence of a flow to the corresponding ODE system thus giving a meaning to the ODE system \eqref{ODE}.

In the case $\eps_N=0$, the general setting of \cite{Am} would require $K \in BV$. That may sometimes be  improved for second order systems like \eqref{ODE}, see \cite{Bo}, \cite{BoCr}, \cite{CJ}, \cite{JaMa}. However for a system in large dimension like \eqref{ODE}, it seems out of reach to obtain renormalized solutions or a well posed flow with only $K\in L^\infty$. Therefore in that case, it is actually {\em critical} to be able to work with only weak solutions to \eqref{Liouville}.

If one had a full diffusion, that is $\Delta_x f_N+\Delta_v f_N$ in the Liouville Eq. \eqref{Liouville}, it would in general be possible to obtain uniqueness together with a flow for the system \eqref{SDE} in some sense , see for instance \cite{Chavanis}, \cite{Fig}, and \cite{LL2}. Note though that even for $\eps_N>0$, the diffusion in \eqref{Liouville} is degenerate (diffusion only in the $v_i$ variables) so that even for $\eps_N> 0$, well posedness for Eq. \eqref{Liouville} does not seem easy with only $K\in L^\infty$.

Of course our analysis also applies to more regular interactions $K$ for which it may be possible to have solutions to the ODE or SDE systems \eqref{ODE} or \eqref{SDE} even if only for short times (a typical example would be $K$ continuous).
%
\subsection{The Vlasov equation (\ref{PDE}). Weak-strong uniqueness}
This article is inspired by a classical weak-strong uniqueness argument for the Vlasov equation, based on the relative entropy of two solutions. Consider two non-negative solutions $f$ and $\tilde f$ with total mass $1$ to Eq. \eqref{PDE}. If $f$ is smooth enough then it is possible to control the distance between them through the relative entropy of $\tilde{f}$ with respect to $f$ or
\[
H(t):=H(\tilde{f}|f)(t) =\int_{\Omega \times \mathbb{R}^d} \tilde{f} \log (\frac{\tilde{f}}{f}) \ud x \ud v .
\]
More precisely, one has the following result
\begin{theorem}[Weak-strong Uniqueness] Assume that $K\in L^\infty$, that  $f(t,x,v)\in L^\infty([0,\ T],\ L^1(\Omega\times\R^d)\cap W^{1,p})$ for any $1 \leq p \leq \infty$ is a strong solution to \eqref{PDE} with 
\begin{equation}
\label{explambda}
\theta_f=\sup_{t\in[0,\ T]}  \int_{\Omega \times \mathbb{R}^d} e^{\lambda |\nabla_v \log f|} f \ud x \ud v < \infty, 
\end{equation}
for some $\lambda>0$. 
Then for any $\tilde f\in L^\infty([0,\ T],\ L^1(\Omega\times\R^d))$, weak solution to (\ref{PDE}) with mass $1$, initial value $\tilde f^0$ and satisfying 
\[
\int_{\Omega\times\R^d} \tilde f\,\log \tilde f \ud x \ud v + \eps \int_0^t\int_{\Omega\times\R^d} \frac{|\nabla_v \tilde f|^2}{\tilde f} \ud x \ud v \ud s\leq \int_{\Omega\times\R^d} \tilde f^0\,\log \tilde f^0 \ud x \ud v,
\]
one has for some constant $C>0$ and any $t\in [0,\ T]$ that as long as $H(\tilde f|\;f)(s)\leq 1$ then for any $s \in [0, t]$
\[
H(\tilde f|\;f)(t)\leq \exp\left(C\,t\,\|K\|_{L^\infty}\,\left(1+\log \theta_f\right)\right)\;H(\tilde f|\;f)(t=0)\,.
\]
In particular if initially $f(t=0)=\tilde f^0$ then $f=\tilde f$ at any later time.\label{weakstrongpde}
\end{theorem}
The short proof of Theorem \ref{weakstrongpde} is given in the appendix and relies on a weighted Csisz\'ar-Kullback-Pinsker (see \cite{BV}).

Theorem \ref{weakstrongpde} requires enough smoothness on $f$. Fortunately such solutions are guaranteed to exist, at least on some bounded time interval per
\begin{Prop}
Assume that $K\in L^\infty$, $f^0\in L^1(\Omega\times \R^d)\cap W^{1,p}$ for every $1 \leq p \leq \infty$ and s.t. for some $\lambda_0>0$
\[
\int_{\Omega \times \mathbb{R}^d} e^{\lambda_0 |\nabla_{(x,v)} \log f^0|} f^0\, dx\,dv < \infty.
\]
Then there exists $T$ depending on $f^0$ and $f\in L^\infty([0,\ T],\ L^1(\Omega\times\R^d)\cap W^{1,p})$ solution to \eqref{PDE} s.t. \eqref{explambda} holds for some $\lambda>0$. Furthermore, if $\eps=0$ and we assume that 
\[
|\nabla_{(x, v)} \log f^0 | \leq  C (1+ |x|^{k}+ |v|^k)
\]
for some $k>0$, 
then 
\begin{equation*}
\sup_{t\in[0,\ T]}|\nabla_{(x, v)} \log f(t, x, v) | \leq C e^{CT} (1+ |x|^k + |v|^k).
\end{equation*}

 \label{strongexist}
\end{Prop} 
\noindent The proof of Prop. \ref{strongexist} is straightforward and also given in the appendix.

It is tempting to try to use directly a result like Theorem \ref{weakstrongpde} to prove the Mean Field limit. In the case of the purely deterministic system \eqref{ODE}, one may associate to each solution the so-called empirical measure $\mu_N$ which is a probability measure on $\Omega\times\R^d$
\begin{equation}
\mu_N(t,x,v)=\frac{1}{N}\sum_{i=1}^N \delta(x-X_i(t))\,\delta(v-V_i(t)).\label{empirical}
\end{equation}
If $(X_i,V_i)_{i=1\dots N}$ solves \eqref{ODE} in an appropriate sense (for instance it comes from a flow), then $\mu_N$ defined through \eqref{empirical} is a solution to Eq. \eqref{PDE} in the sense of distribution. If one could then use a weak-strong uniqueness principle to compare $\mu_N$ to the expected smooth limit $f$ then the Mean Field limit and propagation of chaos would follow. 

This general idea plays an important role in the recent \cite{LazPic} for instance (see also \cite{BoPi,Laz}), leading to an improved truncation parameter (see the discussion after the main result). However Theorem \ref{weakstrongpde} relies on a very different weak-strong uniqueness principle than the one used in \cite{LazPic} and cannot be used directly as it is. There are several reasons for that: In particular Theorem \ref{weakstrongpde} requires the weak solution $\tilde f$ to have a bounded entropy, which cannot be the case of the empirical measure $\mu_N$.

Instead the main result in this article consists in extending Theorem \ref{weakstrongpde} to the Liouville Eq. \eqref{Liouville}.

The study of well-posedness for Vlasov-type systems is now classical and mostly focused on the Vlasov-Poisson case ($K=C\,x/|x|^d$). The existence of weak solutions was obtained in \cite{Arse75} but global existence of strong solutions in dimension $3$ had long been difficult (see \cite{BarDeg} for small initial data) before being obtained in \cite{Scha91}-\cite{Pfaff} and concurrently in \cite{LioPer91} through the propagation of moments (see also \cite{Pallard} for more recent estimates). The most general uniqueness result for the Vlasov-Poisson system was obtained in \cite{Loep06}. 
\subsection{Main result\label{mainresults}}
First define the tensor product of the expected limit $f$ by
\begin{equation*}
\bar{f}_N(t, X, V) = \Pi_{i=1}^N f(t, x_i, v_i),
\end{equation*}
We now compare $f_N$ to $\bar f_N$ through the $N$ dimensional relative entropy  
\begin{equation*}
H_N(f_N \vert \bar{f}_N)(t) = \frac{1}{N} \int_{\Omega^N \times (\mathbb{R}^d)^N}f_N \log(\frac{f_N}{\bar{f}_N}) \ud Z. 
\end{equation*}
We will also write $H_N(t):=H_N(f_N \vert \bar{f}_N)(t)$ in short. 
\begin{theorem}[Propagation of Chaos] 
Assume $K\in L^\infty$  and that $f(t,x,v)\in L^\infty([0,\ T],\ L^1(\Omega\times\R^d)\cap W^{1,p})$ for every $1 \leq p \leq \infty$ solves the Vlasov Eq. \eqref{PDE} with (\ref{explambda}) for some $\lambda>0$. For the case of vanishing randomness, that is in the case $\eps_N \to \eps =0$, we further assume that 
\[
\sup_{t\in[0,\ T]}|\nabla_{(x, v)} \log f(t, x, v) | \leq C (1+ |x|^k + |v|^k).
\] 
Assume that the initial data  $f_N^0$ of the Liouville equation (\ref{Liouville}) satisfies assumptions (\ref{AssLiouWeak1}) and (\ref{AssLiouWeak2}) with $k=1$ and 
\begin{equation*}
H_N(f_N^0\vert \bar{f}_N^0)=\frac{1}{N} \int_{(\Omega \times \mathbb{R}^d)^N}f_N^0 \log(\frac{f_N^0}{\bar{f}_N^0}) \ud Z \to 0, \quad \text{as \ } N \to \infty.
\end{equation*} 
In the case $\eps_N \to \eps =0$, we also assume that 
\[
\sup_{N} \frac{1}{N} \int_{(\Omega \times \mathbb{R}^d)^N} \sum_{i=1}^N \left( 1+ |x_i|^{2k} + |v_i|^{2k}\right) f_N^0 \ud Z < \infty.
\]
For any corresponding weak solution $f_N$ to the Liouville Eq. \eqref{Liouville} as given by Prop. \ref{WeaSolLio} then 
\[
\sup_{t\in [0,\ T]} H_N(f_N\vert \bar{f}_N)(t)\longrightarrow 0, \quad \text{as \ } N \to \infty,
\]
and for any fixed $k$, the $k-$marginal $f_{N, k}$ of $f_N$ converges to the $k-$tensor product of $f$ in $L^1$ as $N  \to \infty$, i.e.
\begin{equation}
\|f_{N,k}-f^{\otimes k}\|_{L^1} \to 0, \quad \text{as\ } N \to \infty. 
\end{equation}\label{propchaos}
\end{theorem}
We recall that the marginals are defined by
\[
f_{N,k}=\int_{(\Omega\times\R^d)^{N-k}} f_N(t,Z) \ud z_{k+1}\dots \ud z_{N}.
\]
Theorem \ref{propchaos} has several consequences
\begin{itemize}
\item It implies a {\em classical Mean Field limit}. First note that the $1$-particle distribution $f_{N,1}$ converges to $f$. Assume that one can obtain solutions to the ODE \eqref{ODE} or SDE \eqref{SDE} system (at least for a short time independent of $N$) for almost all initial data. Consider now  a solution to \eqref{ODE} or \eqref{SDE} with random initial data determined according to the law $f_N^0$; the solution $(X_1(t),V_1(t),\ldots, X_N(t), V_N(t))$ is hence random as well (even the deterministic system \eqref{ODE} propagates any initial randomness). Then the empirical measure as defined by \eqref{empirical} satisfies that
\[
\mathbb{E}\,\mu_N(t,x,v)=f_{N,1}(t,x,v).
\]
Theorem \ref{propchaos} implies that with probability $1$, $\mu_N$ will converge to $f$ for the $weak-*$ topology of measures. We refer to \cite{Gol, GolMouRic, Jab} for a more precise presentation of this connection between the various concepts of Mean Field limit.
\item Theorem \ref{propchaos} is a strong form of propagation of chaos. The usual definition of propagation of chaos typically only require the weak convergence of the marginals, {\em i.e.} for fixed $k$ 
\[
f_{N,k} \longrightarrow f^{\otimes k}\quad\mbox{in}\ {\cal D}'.
\] 
Here we not only have strong convergence in $L^1$ for all marginals but also an explicit bound on the distance from the full law $f_N$. 

Such stronger notions of propagation of chaos have recently been more thoroughly investigated and some of the connections between them elucidated, we refer to \cite{HM, MM, MMW} for example or to the survey \cite{Jab}.
\item It is possible to be even more precise on the convergence of the marginals and in fact, one controls the relative entropy of each of them as
\begin{equation}
\begin{split}
H_k(f_{N,k}|\;f^{\otimes k})&=\frac{1}{k}\int_{(\Omega\times\R^d)^k} f_{N,k}\,\log \left(\frac{f_{N,k}}{f^{\otimes k}}\right)\,dz_1\dots dz_k\\ 
&\leq H_N(f_N|\;\bar f_N)(t)
\longrightarrow 0.
\end{split}\label{controlhk}
\end{equation}
The fact that the scaled entropy $H_N$ actually controls any other scaled entropy $H_k$ is {\em critical} for that and for the conclusion of Theorem \ref{propchaos}. We refer to the references above for the proof of this inequality.
\item Theorem \ref{propchaos} is quite demanding on the expected limit $f$, in particular through assumption \eqref{explambda}. This is in line with the assumption \eqref{explambda} of Theorem \ref{weakstrongpde} and with the general idea of weak-strong estimates: The weak requirements on $f_N^0$ and $K$ are replaced by strong assumptions on the limit. The assumption \eqref{explambda} is satisfied if $f$ has Gaussian or any kind of exponential decay: $f\sim e^{-\nu\,|v|^\alpha}$. In general $C^k$ functions with compact support cannot satisfy \eqref{explambda} though Gevrey-like regularity seems to be possible.
\item Theorem \ref{propchaos} is really a conditional result: It holds on any time interval $[0,\ T]$ for which one has existence of an appropriate solution $f$ to the Vlasov Eq. \eqref{PDE}. Prop. \ref{strongexist} guarantees that such a time interval will exist but $T$ could be larger than what is given by Prop. \ref{strongexist}. One may very well have $T=+\infty$ for some initial data or if additional regularity is known for $K$.
\item It would be relatively straightforward to extend Theorem \ref{propchaos} to 1st order systems of the kind
\begin{equation}
\ud X_i=\frac{1}{N}\sum_{j=1}^N K(X_i-X_j)\ud t +\eps_N \ud W_i^t,\label{1storderSDE}
\end{equation}
provided that appropriate assumptions are made on $K$, in particular that $\mbox{div}\,K=0$.
\item The estimates underlying Theorem \ref{propchaos} are in fact explicit, see subsection \ref{proofchaos}. This allows to handle interaction kernels $K_N$ depending on $N$ provided that $\sup_N \|K_N\|_{L^\infty}<\infty$. This is typical of numerical settings (particle' methods for instance) where $K$ is typically truncated or regularized. 
\end{itemize}

\bigskip

The first proofs of the Mean Field limit for deterministic systems such as \eqref{ODE} were performed in \cite{BraHep77, Dobr79, Neun79} (see also \cite{Spoh91}). Those now classical results have introduced the main concepts and questions for the Mean Field limit and propagation of chaos. They demand that $K\in W^{1,\infty}$ and rely on the corresponding Gronwall estimates for systems of ODEs (extended to infinite dimensional settings).

Obviously $K\in W^{1,\infty}$ is an important limitation which does not allow to treat many interesting kernels, either from the physics point of view or for numerical methods. In that last case it often makes sense to regularize or truncate $K$. Since in many settings, $K$ is only singular at the origin, $x=0$, this leads for instance to working with a smooth $K_N$ s.t. $K_N(x)=K(x)$ for $|x|\geq \eps_N$; $\eps_N$ being some determined scale which typically vanishes with $N$. The accuracy of the method depends on how small the scale $\eps_N$ can be taken; one critical scale is $\eps_N=N^{-1/d}$ which would be the minimal distance in physical space of $N$ particles over a grid.  

For Poisson kernels, $K=C\,x/|x|^d$, the Mean Field limit was obtained for particles initially on a regular mesh in \cite{Victory}, \cite{Wollman} for $\eps_N>>N^{-1/d}$. When the particles are not initially regularly distributed, propagation of chaos was obtained in \cite{Victory2} but only for $\eps_N\sim (\log N)^{-1}$. 
As mentioned above, those results were recently improved in \cite{LazPic} with much smaller truncation scales $\eps_N<<N^{-1/d}$.

The only results for deterministic second order systems with singular (non-Lipschitz) kernels without truncation are \cite{HauJab07} and the more recent \cite{HauJab14} for the propagation of chaos. Those require that $K$ satisfies for some $\alpha<1$
\[
|K(x)|\leq \frac{C}{|x|^\alpha},\quad |\nabla K(x)|\leq \frac{C}{|x|^{\alpha+1}}.
\] 
The result presented here does not require any regularity on $K$ (any bound on $|\nabla K|$) but does not allow $K$ to be unbounded either. It is therefore not directly comparable. In fact Theorem \ref{propchaos} is interesting precisely because it introduced a new and unexpected critical scale, $K\in L^\infty$.

\medskip

The derivation of the Mean Field limit and the propagation of chaos is more advanced for 1st order deterministic systems (System \eqref{1storderSDE} with $\eps_N=0$ for instance). Systems like \eqref{1storderSDE} with a kernel $K$ non smooth only at the origin $x=0$ enjoy additional symmetries with respect to second order which makes the derivation easier.  We refer to \cite{Jab} for a more thorough comparison. 

The main example of such 1st order system is the point vortex method for the 2D Euler equations. The Mean Field limit has been obtained for well distributed initial conditions, see for example \cite{CotGooHou, GooHouLow90, HouLowShe93} while the proof of propagation of chaos can be found in \cite{Scho95, Scho96}. We refer to \cite{Hau09} for the best results so far for general multi-dimensional 1st order systems.

\medskip

In comparison with the deterministic case, the stochastic case, $\eps_N>0$ in \eqref{SDE} or \eqref{1storderSDE}, seems harder as many of the techniques developed in the deterministic setting are not applicable. The Lipschitz case, $K\in W^{1,\infty}_{loc}$ can still be handled through Gronwall like inequalities, see for instance \cite{CarCanBol, CCH}. 

In the non degenerate case, $\eps_N\rightarrow \eps>0$ in \eqref{1storderSDE} for instance, then the regularizing properties of the stochastic part can actually be exploited to handle some singularity in $K$ (up to order $1/|x|$). For 1st order systems, propagation of chaos can hence be proved for the 2D viscous or stochastic vortex systems for the Euler equations, leading to the 2D incompressible Navier-Stokes system; see \cite{FlGuPr, FHM-JEMS, Osada}.
 
However the system considered here \eqref{SDE} has a degenerate stochastic part (there is no diffusion in the $x$ variable) which may in addition vanish at the limit if $\eps_N\rightarrow 0$. Theorem \ref{propchaos} is {\em the only result} that we are aware of in such a degenerate setting for non Lipschitz force terms. 
\medskip
\subsection{From combinatorics to Theorem \ref{propchaos} } \label{proofchaos}
 Define for any $p\geq 1$
\[
M_p :=\left( \int_{\Omega \times \mathbb{R}^d } |\nabla_v \log f|^p f\ud x \ud v \right)^{\frac{1}{p}}.
\] 
Theorem \ref{propchaos} is a straightforward consequence of 
\begin{theorem}
\label{Main Theorem}
Assume that $f\in L^\infty\cap L^1(\Omega\times\R^d)$ with $f\geq 0$ and $\int f=1$, that $\nabla_v f\in W^{1,p}_{loc}$ for every 
$1 \leq p \leq \infty$ with $\sup_{p<\infty} \frac{M_p}{p} < \infty $ and that $\|K\|_{L^\infty} \left( \sup_p \frac{M_p}{p} \right)  < \frac{1}{8e^2}$, then 
\begin{equation*}
\int_{(\Omega \times \mathbb{R}^d)^N} \bar{f}_N \exp(|R_N|) \ud Z   \leq 5 + 6\left( \frac{8e^2 \|K\|_{L^\infty} \left( \sup_p \frac{M_p}{p}\right)}{ 1- \left(8e^2 \|K\|_{L^\infty} \left( \sup_p \frac{M_p}{p}\right) \right)^2}\right)^2 < \infty,
\end{equation*}
where $\bar{f}_N=\Pi_{i=1}^N f(t,x_i,v_i)$ and $R_N$ is defined by
\begin{equation}
\label{RN}
R_N= \frac{1}{N}\sum_{i, j=1}^N \nabla_{v_i} \log f(x_i, v_i) \cdot \left\{ K(x_i-x_j)  -K\star \rho(x_i)\right\}. 
\end{equation}
\end{theorem}
The proof of Theorem \ref{Main Theorem} is the main technical difficulty of the article and will be done in the next section. Instead we explain here how to simply prove Theorem \ref{propchaos} assuming Theorem \ref{Main Theorem}.

First of all we observe that the assumption $ \sup_p \frac{M_p}{p} < \infty $ is essentially  equivalent to the assumption \eqref{explambda} in Theorem \ref{propchaos}.
Indeed,  
\begin{itemize}
\item[i)] $\sup_p \frac{M_p}{p} <  \Lambda $ implies $\int f e^{ \lambda |\nabla_v \log f|}  \ud z $ is finite for any $\lambda < \frac{1}{e \Lambda}$: By Taylor expansion for $e^x$, 
\[
\begin{split}
\int f e^{ \lambda |\nabla_v \log f|} \ud z &\leq 1 + \sum_{p=1}^\infty \frac{1}{p!} \lambda^p \int f |\nabla_v \log f|^p \ud z \\
& \leq 1 + \sum_{p=1}^\infty \frac{1}{p!} \lambda^p  (\Lambda p)^p \leq 1+ \sum_{p=1}^\infty \left( e\Lambda\lambda\right)^p.
\end{split}
\]
 
\item[ii)] Assumption (\ref{explambda}) implies $\sup_p \frac{M_p}{p} \leq \Lambda$,  where $\Lambda$  depends on the integral value $\int f e^{ \lambda| \nabla_v \log f|}\ud z $.  Indeed, for any $p=1, 2, \cdots$, 
\[
\int f |\nabla_v \log f|^p \ud z \leq p! \lambda^{-p} \int f e^{\lambda |\nabla_v \log f|} \ud z.
\]
Since $p! \leq p^p$, 
\[
\sup_{p} \frac{M_p}{p} \leq  \frac{1}{\lambda} \sup_p (\int f e^{ \lambda |\nabla_v \log f|} \ud z)^{\frac{1}{p}} < \infty. 
\]
\end{itemize}

\medskip

Now recall that $f$ is a strong solution to the Vlasov Eq. \eqref{PDE}. Therefore $\bar f_N$ solves 
\begin{equation}
\label{NProduct}
\partial_t \bar{f}_N + L_N \bar{f}_N = \bar{f}_N R_N + (\eps- \eps_N) \sum_{i=1}^N \Delta_{v_i} \bar{f}_N,
\end{equation}
where  
\begin{equation}
R_N= \sum_{i=1}^N \left\{ \frac{1}{N} \sum_{j \ne i} K(x_i-x_j) \cdot \nabla_{v_i} \log f(x_i, v_i) -K \star \rho(x_i) \cdot \nabla_{v_i} \log f(x_i, v_i) \right\}.
\end{equation}
With the convention that $K(0) =0$, this is equivalent to the definition \eqref{RN}.

From this point the initial calculations exactly follow the proof of Theorem \ref{weakstrongpde} as given in the appendix.
Since $f_N$ is a weak solution to the Liouville Eq. according to Prop. \ref{WeaSolLio}
\[\begin{split}
H_N(t)&=\frac{1}{N}\int_{(\Omega \times \mathbb{R}^d)^N} f_N \log (\frac{f_N}{\bar f_N}) \ud Z=\frac{1}{N}\int  f_N\,\log f_N-\frac{1}{N}\int  f_N\log \bar f_N\\
&\leq \frac{1}{N}\int  f^0_N\,\log     f^0_N-\frac{\eps_N}{N}\int_0^t \int\frac{|\nabla_V  f_N|^2}{f_N}-\frac{1}{N}\int  f_N\log \bar f_N,
\end{split}\]
 per the assumption of dissipation of entropy for $f_N$ in Prop. \ref{WeaSolLio}.

Since $\bar f_N$ is smooth, $\log \bar{f}_N$ can be used as a test function  against $f_N$ which is a weak solution to the Liouville Eq. \eqref{Liouville} so that
\[\begin{split}
&\int f_N\, \log \bar f_N=\int f^0_N\,\log \bar f^0_N+\int_0^t\!\int f_N (s,X,V)\,(\partial_t \log \bar f_N+L_N^*\,\log \bar f_N)\ud Z \ud s,
\end{split}
\]
where 
\begin{equation*}
L_N^* =\sum_{i=1}^Nv_i \cdot \nabla_{x_i} + \frac{1}{N}\sum_{i=1}^N \sum_{j \ne i} K(x_i-x_j) \cdot \nabla_{v_i}+\eps_N\sum_{i=1}^N \Delta_{v_i}.
\end{equation*}
Since $\bar f_N$ is a strong solution to \eqref{NProduct}, this leads to 
\begin{equation*}
\begin{split}
&\int f_N\, \log \bar f_N=\int f^0_N\,\log \bar{f}^0_N+\int_0^t\!\int  f_N\,R_N \ud Z \ud s\\
+&\eps_N\,\int_0^t\!\int f_N\,\left(\frac{\Delta_V \bar f_N}{\bar f_N}+ \Delta_V \log \bar f_N\right) \ud Z \ud s + (\eps -\eps_N) \int_0^t \int f_N \frac{\Delta_V \bar{f}_N}{\bar{f}_N} \ud Z \ud s.
\end{split}
\end{equation*}
Hence, 
\begin{equation}
\label{RelativeEntropy}
\begin{split}
H_N(t)  &\leq H_N(0) -\frac{1}{N} \int_0^t \int f_N R_N  \ud Z \ud s \\&-\frac{\eps_N}{N} \int_0^t \int \left[ \frac{|\nabla_V f_N|^2}{f_N} +f_N \left( \frac{\Delta_V \bar{f}_N}{\bar{f}_N}+ \Delta_V \log \bar{f}_N\right)\right] \ud Z \ud s \\
&-\frac{\eps-\eps_N}{N}  \int_0^t \int f_N \frac{\Delta_V \bar{f}_N}{\bar{f}_N} \ud Z \ud s .
\end{split}
\end{equation}
We now treat the three types of the choices of $\eps_N$ separately. 

\smallskip

\noindent\textbf{Case I}:  $\eps_N =\eps \geq 0$. In this case, the last term in the right-hand side of \eqref{RelativeEntropy} vanishes. Classical entropy estimates show that
\[
\int\frac{|\nabla_V  f_N|^2}{f_N}+\int f_N\,\left(\frac{\Delta_V \bar f_N}{\bar f_N}+ \Delta_V \log \bar f_N\right) \ud Z \geq 0,
\]
see the appendix for detailed calculations.

Therefore we finally obtain that
\begin{equation}
\label{EvolutionofEntropy}
\begin{split}
 H_N(t) &\leq  H_N(0)-\frac{1}{N} \int_0^t\int f_N R_N \ud Z\ud s.\\
  \end{split}
\end{equation}
\smallskip

\noindent\textbf{Case II}: $\eps_N \to \eps >0$. The terms in \eqref{RelativeEntropy} induced by randomness can be bounded by the entropy of $f_N^0$,
\[
\begin{split}
& -\frac{1}{N} \int_0^t \int \left[ \eps_N \frac{|\nabla_V f_N|^2}{f_N} + \eps f_N  \frac{\Delta_V \bar{f}_N}{\bar{f}_N}+  \eps_N f_N \Delta_V \log \bar{f}_N\right] \ud Z \ud s \\
&\qquad= -\frac{1}{N} \int_0^t \int f_N \eps | \nabla_V \log \bar{f}_N -\frac{\eps+ \eps_N }{2\eps} \nabla_V \log f_N|^2 \ud Z \ud s\\
&\qquad\quad + \frac{(\eps-\eps_N)^2}{4\eps} \int_0^t \int \frac{|\nabla_V f_N|^2}{f_N} \ud Z \ud s 
\\
&\qquad \leq \frac{(\eps-\eps_N)^2}{4\eps}\frac{1}{N} \int_0^t \int \frac{|\nabla_V f_N|^2}{f_N} \ud Z \ud s   \leq  \frac{(\eps-\eps_N)^2}{4 \eps \eps_N } \frac{1}{N} \int f_N^0 \log f_N^0 .
\end{split}
\]
Therefore, we obtain that 
\begin{equation}
\label{EvolutionofEntropyCor1}
\begin{split}
 H_N(t) &\leq  H_N(0)-\frac{1}{N} \int_0^t\int f_N R_N \ud Z\ud s + \alpha_N,\\
  \end{split}
\end{equation}
where $\alpha_N \to 0$ as $N \to \infty$. 

\smallskip

\noindent\textbf{Case III}: $\eps_N \to \eps =0$. This is the vanishing randomness case, that is  there is no diffusion in the limit Vlasov equation. The terms in \eqref{RelativeEntropy} induced by randomness in $N-$particle system can also be bounded but by some moment bounds for $f_N^0$,
\[
\begin{split}
& S(\eps_N):= -\frac{\eps_N}{N} \int_0^t \int \left[\frac{|\nabla_V f_N|^2}{f_N} + f_N \Delta_V \log \bar{f}_N\right] \ud Z \ud s\\
=& -\frac{\eps_N}{N} \int_0^t \int \frac{|\nabla_V f_N|^2}{f_N} + \frac{\eps_N}{N} \int_0^t \int \nabla_V f_N  \cdot \nabla_V \log \bar{f}_N \\
\leq &\ \frac{\eps_N}{4N } \int_0^t \int  f_N |\nabla_V \log \bar{f}_N|^2 \ud Z \ud s. 
\end{split}
\]
This is the why we add here extra moment restrictions. Recall  that 
\[
|\nabla_v \log f | \leq |\nabla \log f | \leq C (1+ |x|^k + |v|^k|).
\] 
Therefore,  
\[
\begin{split}
 S(\eps_N)\leq  \frac{\eps_N}{4} \left(\frac{1}{N} \int_0^t \int \sum_{i=1}^N (1 +|x_i|^{2k} +|v_i|^{2k}) f_N \ud Z \ud s  \right) \to 0,
\end{split}
\]
as $N \to \infty$. Hence, we also obtain \eqref{EvolutionofEntropyCor1} in this case with $\alpha_N \to 0$ as $N \to \infty$. 

\medskip

Now we can proceed to prove the estimate for $H_N(t)$. 
Recall the Frenchel's inequality for the function $u(x) =x \log x$: For all $x,\;y\geq 0$ 
\[
xy \leq x \log x + \exp(y-1).
\]
Hence for  $\nu>0$
\begin{equation*}
-f_N R_N \leq \frac{\bar{f}_N}{\nu}\left(  \frac{f_N}{\bar{f}_N} \, \nu\,|R_N| \right) \leq \frac{\bar{f}_N}{\nu}  \left( \frac{f_N}{\bar{f}_N} \log(\frac{f_N}{\bar{f}_N}) + \exp(\nu\,|R_N|)\right).
\end{equation*}
Therefore 
\begin{equation}
\label{Evolution}
H_N(t) \leq  H_N(0)+\alpha_N  +\frac{1}{\nu}\int_0^t H_N(s)\,\ud s
+ \frac{1}{\nu}\frac{1}{N}\int_0^t  \int \bar{f}_N \exp(\nu\,|R_N|)  \ud Z\,\ud s. 
\end{equation}
Now define $\tilde K=\nu\, K$ and take $\nu$ s.t. 
\[
\|\tilde K\|_{L^\infty}\,\sup_p \frac{M_p}{p}=\nu\,\|K\|_{L^\infty}\,\sup_p \frac{M_p}{p}\leq\frac{1}{16\,e^2}.
\]
We may apply Theorem \ref{Main Theorem} to $\tilde K$ and $\tilde R_N=\nu\,R_N$. This implies that 
\[
L=\sup_N \sup_{t\in [0,\ T]} \int \bar{f}_N \exp(\nu\,|R_N|)  \ud Z<\infty.
\]
Inserting this in \eqref{Evolution} gives
\[
H_N(t) \leq  H_N(0)+ \alpha_N   + \frac{1}{\nu}\int_0^t H_N(s)\,\ud s+\frac{L t}{\nu N},
\]
and up to time $T >0$, by Gronwall's inequality   
\begin{equation}
\label{estimateofentropy1}
H_N(f_N\vert \bar{f}_N)(t) \leq \left( H_N(f_N\vert \bar{f}_N)(0) + \alpha_N +\frac{L T}{\nu N} \right)\; \exp(t/\nu),
\end{equation}
which gives the first part of Theorem \ref{propchaos}.

Next apply the estimates in \cite{HM}, \cite{MM} and \cite{MMW}. In particular
by the properties of relative entropy functional, we have for any fixed $k \geq 1$, 
\[
H_k(f_{N,k} \vert f^{\otimes k})= \frac{1}{k} \int_{(\Omega \times \mathbb{R}^d)^k} f_{N, k} \log\left( \frac{f_{N,k}}{f^{\otimes k}}\right)\ud z_1 \cdots \ud z_k
\leq H_N(f_N \vert \bar{f}_N) \longrightarrow 0, 
\]
as $N \to \infty$. 

The classical  Csisz\'ar-Kullback-Pinsker inequality (see chapter 22 in \cite{V}) then implies that
\[
\|f_{N,k}-f^{\otimes k}\|_{L^1}  \leq \sqrt{2k H_k(f_{N, k}\vert f^{\otimes k})} \to 0 
  \]
as $N \to \infty$. 
This completes the proof of Theorem \ref{propchaos}. 
\subsection{The scaling of $R_N$} 
The full proof of Theorem \ref{Main Theorem} is given in the next section but we present here some of the basic scaling properties of $R_N$.

A trivial bound for $|R_N|$ is simply
\begin{equation}
|R_N| \leq (2 \|K\|_{L^\infty} \|\nabla_v \log f\|_{L^\infty})\, N. 
\end{equation}
However inserting this bound in \eqref{Evolution} would only give that $H_N(t)=O(1)$ without any chance of converging. Instead Theorem \ref{Main Theorem} essentially proves that $R_N$ is of order $1$ and not of order $N$.

To get 
\[
\int_{(\Omega \times \mathbb{R}^d)^N} \bar{f}_N \exp(|R_N|) \ud Z \leq C < \infty,
\]
where $C$ doesn't depend on $N$,  we expand $\exp(|R_N|)$ by Taylor expansion.  Note though that  
\[ 
\begin{split}
&\frac{1}{(2k+1)!} |R_N|^{2k+1} \leq \frac{1}{(2k+1)!} |R_N|^{2k} \left(  \frac{2k+1}{2} + \frac{1}{2(2k+1)} |R_N|^2 \right)\\
\leq & \frac{1}{2} \frac{1}{(2k)!} |R_N|^{2k} + \frac{1}{(2k+2)!} |R_N|^{2k+2},
\end{split}
\]
so that we only have to bound the even terms and have 
\begin{equation*}
\exp(|R_N|)=\sum_{k=0}^\infty \frac{1}{k!} |R_N|^k \leq 3 \sum_{k=0}^\infty \frac{1}{(2k)!} |R_N|^{2k}. 
\end{equation*}
Consequently, we have 
\begin{equation}
\label{eventerms}
\int \bar{f}_N \exp(|R_N|) \ud Z \leq 3 \sum_{k=0}^{\infty} \frac{1}{(2k)!} \int |R_N|^{2k} \bar{f}_N \ud Z.
\end{equation}
The basic idea of the proof for Theorem \ref{Main Theorem} is to expand the sum defining $R_N$ in $R_N^{2k}$ and show that a large number of terms vanish under integral with respect to $\bar{f}_N$. 

For the moment we just present two basic calculations, indicative of the type of cancellations that we use
\begin{lemma}
\label{vanish1} Assume that $f\in L^\infty\cap L^1(\Omega\times\R^d)$ with $f\geq 0$ and $\int f=1$.
Assume  that $K\in L^\infty$ and that $\nabla_v  f \in L_{\text{loc}}^1 $, then 
\begin{equation*} 
\int_{\Omega^N \times (\mathbb{R}^d)^N} R_N \bar{f}_N \ud Z =0. 
\end{equation*}
\end{lemma}
\begin{proof} Simply expanding $R_N$, we get 
\[\begin{split}
\int R_N \bar{f}_N  \ud Z =\frac{1}{N} \sum_{i, j=1}^N \int \nabla_{v_i}\log f(x_i, v_i) \cdot &\{K(x_i-x_j)-K\star\rho(x_i)\}\\
& f(x_1, v_1) \cdots f(x_N, v_N) \ud Z.
\end{split}\]
For fixed $(i,j)$, notice that $f(x_i, v_i) \nabla_{v_i} \log f(x_i, v_i)= \nabla_{v_i} f(x_i, v_i)$, and no other terms depend on $v_i$. Integration by parts  thus implies that the integral vanishes. Indeed, by Fubini's Theorem,  without loss of generality, we only need to check 
\begin{equation}
\label{L1I1}
\int \nabla_{v} f(x, v)\, K\star\rho(x, v) \ud x \ud v =0
\end{equation}
and  
\begin{equation}
\label{L1I2}
\int \nabla_{v_1} f(x_1, v_1) \{K(x_1-x_2)- K\star \rho(x_1) \} \rho(x_2) \ud x_1 \ud v_1 \ud x_2=0.
\end{equation}
Both \eqref{L1I1} and \eqref{L1I2} are easily proved by truncating the integral with some $\varphi_L$ such as
\[
\varphi_L(x, v) =
\begin{cases}
1, &\text{\  if }  |(x, v)| \leq L, \\
\in (0, 1) &\text{\ if }  L <|(x,v)|\leq 2L, \\
0, &\text{\ if } |(x, v)| > 2L,
\end{cases}
\]
and letting $L$ go to $\infty$.
%
%

\end{proof}
Lemma \ref{vanish1} only illustrates the simplest cancellation in $R_N$. It is also straightforward to show some orthogonality property between the terms in the sum defining $R_N$. This leads to the first indication that indeed $R_N$ is of order $1$ and not $N$. 
\begin{lemma} 
\label{vanish2}
Assume that $f\in L^\infty\cap L^1(\Omega\times\R^d)$ with $f\geq 0$ and $\int f=1$.
Assume  that $K\in L^\infty$ and that $\nabla_v  f \in L_{\text{loc}}^2 $, then  
\begin{equation*}
\int_{\Omega^N \times (\mathbb{R}^d)^N} |R_N|^2 \bar{f}_N \ud Z \leq 4 \|K\|_{L^\infty}^2 \int_{\Omega \times \mathbb{R}^d} |\nabla_v \log f|^2 f \ud x \ud v.
\end{equation*}
\end{lemma}

\begin{proof}
For convenience we denote
\[
F_i=\nabla_{v_i} \log f(x_i, v_i),\quad k_{i, j}=K(x_i-x_j) -K\star \rho(x_i).
\]
Simply expand the left-hand side 
\[
\int|R_N|^2 \bar{f}_N \ud Z 
=\frac{1}{N^2}\sum_{i_1, i_2=1}^N \sum_{j_1, j_2=1}^N \int F_{i_1} \cdot k_{i_1, j_1}  F_{i_2} \cdot k_{i_2, j_2} \bar{f}_N \ud Z. 
\]
If $i_1 \ne i_2$, then by integration by parts, 
\[
\int F_{i_1} \cdot k_{i_1, j_1}  F_{i_2} \cdot k_{i_2, j_2} \bar{f}_N \ud Z=0. 
\] 
Indeed, without loss of generality, let $i_1=1$ and $i_2=2$, then 
\[\begin{split}
&\int F_{i_1} \cdot k_{i_1, j_1}  F_{i_2} \cdot k_{i_2, j_2} \bar{f}_N \ud Z\\
&\qquad =\int_{(\Omega \times \mathbb{R}^d)^2} \nabla_{v_1} f(x_1, v_1)\cdot k_{1, j_1} \nabla_{v_2} f(x_2, v_2) \cdot k_{2, j_2} \ud z_1 \ud z_2 =0,
\end{split}\]
by integration by parts since $k_{1, j_1}$ and $k_{2, j_2}$ do not depend any $v$ variables. 

If $i_1=i_2$ while $j_1 \ne j_2$, then at least one of $\{j_1, j_2\}$ is not equal to $i_1$, then this type of integral vanishes by the definition of convolution. Indeed, without lost of generality, let assume that $i_1=i_2=1$ and $j_1=2$ while $j_2 \ne 2$, then 
\[
\begin{split}
&\int_{(\Omega \times \mathbb{R}^d)^N} F_{i_1} \cdot k_{i_1, j_1}  F_{i_2} \cdot k_{i_2, j_2} \bar{f}_N \ud Z \\
&\ =
\int_{(\Omega \times \mathbb{R}^d)^N} \left[ \nabla_{v_1} \log f(x_1, v_1) \cdot \{K(x_1-x_2)-K\star\rho(x_1)\} \right] \\ & \qquad \cdot  \left[ \nabla_{v_1} \log f(x_1, v_1) \cdot \{K(x_1-x_{j_2})-K\star\rho(x_1)\} \right] \bar{f}_N \ud Z  \\
&\ = \int_{(\Omega \times \mathbb{R}^d)^{N-1}} \left[ \nabla_{v_1} \log f(x_1, v_1) \cdot \{K(x_1-x_{j_2})-K\star\rho(x_1)\} \right] \Pi_{i \ne 2} f(x_i, v_i ) \ud z_i  \\
&\qquad \cdot \left( \nabla_{v_1} \log f(x_1, v_1) \cdot  \int_{\Omega} \{K(x_1 -x_2)-K\star\rho(x_1)\} \rho(x_2) \ud x_2 \right)\\
&\ =0,
\end{split}
\]
where we used that    
\[
\int_{\Omega} \{K(x_1 -x_2)-K \star \rho(x_1)\} \rho(x_2) \ud x_2 =0,
\]
by  the definition of convolution, and since $\rho$ has integral $1$.

Hence after integration only those terms with indices $i_1=i_2$ and $j_1 =j_2$ contribute to the summation. That is  
\[
\begin{split}
&\frac{1}{N^2}\sum_{i_1, i_2=1}^N \sum_{j_1, j_2=1}^N \int F_{i_1} \cdot k_{i_1, j_1}  F_{i_2} \cdot k_{i_2, j_2} \bar{f}_N \ud Z\\
=& \frac{1}{N^2} \sum_{i=1}^N \sum_{j=1}^N \int (F_i \cdot k_{i,j})^2 f_N \ud Z \leq 4 \|K\|^2_{L^\infty} \int_{\Omega \times \mathbb{R}^d} |\nabla_v \log f |^2 f \ud x \ud v,
\end{split}
\] 
which completes the proof.
\end{proof}
\section{Main Estimates: Proof of Theorem \ref{Main Theorem} \label{Main Estimate}}
%
From the remark \eqref{eventerms}, it is enough to bound
\[
\sum_{k=0}^{\infty} \frac{1}{(2k)!} \int |R_N|^{2k} \bar{f}_N \ud Z 
\] 
which we divide in two different cases: $k$ is small compared to $N$ or $k$ is comparable or larger than $N$. The first part,  $3k \leq N$, is more delicate and requires some preparatory combinatorics work. The second part,  $3k >N$, is almost trivial since now the coefficients $\frac{1}{(2k)!}$ dominates. The trivial bound for $|R_N|$ is good enough in this case. 

Accordingly Theorem \ref{Main Theorem} is a consequence of the following two propositions
\begin{Prop}
\label{ksmall}
For $3k \leq N$, we have \begin{equation*}
\sum_{k=0}^{\lfloor \frac{N}{3}\rfloor } \frac{1}{(2k)!} \int |R_N|^{2k} \bar{f}_N \ud Z \leq 1 + 2\sum_{k=1}^{\lfloor \frac{N}{3}\rfloor } k \left(8e^2 \|K\|_{L^\infty} \left(\sup_p \frac{M_p}{p} \right) \right)^{2k}.
\end{equation*}

\end{Prop}
\begin{Prop}
\label{kbig}
For $3k > N$, we have 
\begin{equation*} 
\sum_{k=\lfloor \frac{N}{3}\rfloor+1}^\infty \frac{1}{(2k)!} \int |R_N|^{2k} \bar{f}_N \ud Z \leq \sum_{k=\lfloor \frac{N}{3}\rfloor+1}^\infty \left(5e^2 \|K\|_{L^\infty}\left( \sup_p \frac{M_p}{p}\right) \right)^{2k}. 
\end{equation*}
\end{Prop}
Let us briefly explain how we can  prove Theorem \ref{Main Theorem} from Proposition  \ref{ksmall} and Proposition  \ref{kbig}.  
 
\begin{proof}[Proof of Theorem \ref{Main Theorem}] Recall that
\[
\sum_{k=1}^\infty k\, r^k=r\,\frac{d}{dr}\sum_{k=0}^\infty r^k=\frac{r}{(1-r)^2}. 
\]
Under the assumption $\|K\|_{L^\infty} \sup_p \frac{M_p}{p} < \frac{1}{8e^2}$, we have that
\[\begin{split}
\sum_{k=1}^{\lfloor \frac{N}{3}\rfloor } k \left(8e^2\|K\|_{L^\infty} \left(\sup_p \frac{M_p}{p}\right)\right)^{2k} &\leq \sum_{k=1}^{\infty } k \left(8e^2\|K\|_{L^\infty} \left(\sup_p \frac{M_p}{p}\right)\right)^{2k}\\
&=\frac{\left( 8e^2 \|K\|_{L^\infty} \left(\sup_p \frac{M_p}{p}\right)\right)^2}{\left(1-\left( 8e^2 \|K\|_{L^\infty} \left(\sup_p \frac{M_p}{p}\right)\right)^2\right)^2} < \infty,
\end{split}\]
and 
\[
\sum_{k=\lfloor \frac{N}{3}\rfloor+1}^\infty \left(5e^2 \|K\|_{L^\infty} \left(\sup_p \frac{M_p}{p}\right)\right)^{2k} \leq \sum_{k=1}^\infty \left(\frac{5}{8} \right)^{2k} \leq \frac{\left(\frac{5}{8} \right)^2}{1- \left(\frac{5}{8} \right)^2}< \infty .
\]
Hence, by (\ref{eventerms}), Proposition  \ref{ksmall} and Proposition  \ref{kbig} we have that 
\[
\begin{split}
& \int \bar{f}_N \exp(|R_N|) \ud Z \leq 3 \sum_{k=0}^{\infty} \frac{1}{(2k)!} \int |R_N|^{2k} \bar{f}_N \ud Z \\
\leq & 3 \left( 1+ 2 \sum_{k=1}^{\infty } k \left(8e^2\|K\|_{L^\infty} \left(\sup_p \frac{M_p}{p}\right)\right)^{2k}+\sum_{k=1}^\infty \left(\frac{5}{8}\right)^{2k} \right)\\
= & 3\left( 1+ \frac{2\left( 8e^2 \|K\|_{L^\infty} \left(\sup_p \frac{M_p}{p}\right)\right)^2}{\left(1-\left( 8e^2 \|K\|_{L^\infty} \left(\sup_p \frac{M_p}{p}\right)\right)^2\right)^2} + \frac{\left(\frac{5}{8} \right)^2}{1- \left(\frac{5}{8} \right)^2}\right) \\
\leq & 5 + 6\left( \frac{8e^2 \|K\|_{L^\infty} \left( \sup_p \frac{M_p}{p}\right)}{ 1- \left(8e^2 \|K\|_{L^\infty} \left( \sup_p \frac{M_p}{p}\right) \right)^2}\right)^2 .
\end{split}
\]
This completes the proof. 
\end{proof}

We now proceed to establish the above propositions. For convenience we will keep on using the notations of Lemma \ref{vanish2}
\[
F_i=\nabla_{v_i} \log f(x_i, v_i),\quad k_{i, j}=K(x_i-x_j) -K\star \rho(x_i).
\]

\subsection{The case $3k \leq  N$: Proof of Proposition \ref{ksmall}} 
We start with the general rule for cancellation in $R_N$
\begin{lemma} [General Cancellation Rule]
\label{L3}
\label{generalcancelationrule} Fix an integer $p \geq 1$. Take any pair of multi-indices $(I_p, J_p)$, where $I_p=(i_1, i_2, \cdots, i_p)$ and $J_p=(j_1, j_2, \cdots, j_p)$. All components of $I_p$ and $J_p$ are taken from the set $\{1, 2, \cdots, N\}. $ Then 
\begin{equation}
\label{integral_indicate}
\begin{split}
&\int_{(\Omega\times \mathbb{R}^d)^N } \left( \nabla_{v_{i_1}} \log f(x_{i_1}, v_{i_1} )\cdot \{K(x_{i_1}-x_{j_1})-K*\rho(x_{i_1})\} \right) \\ &\cdots \left(\nabla_{v_{i_p}} \log f(x_{i_p}, v_{i_p} )\cdot \{K(x_{i_p}-x_{j_p})-K*\rho(x_{i_p})\} \right)\bar{f}_N \ud Z =0
\end{split}
\end{equation}
 provided that one of the following statements is satisfied: \\
1) there exists one $i_\nu$,   such that $i_\nu \notin \{i_1, \cdots, i_{\nu-1}, i_\nu, \cdots, i_p \}$; \\ 2) there exists one $j_\nu$, such that $j_\nu \notin \{i_1, i_2, \cdots, i_p \}\cup \{j_1, \cdots, j_{\nu-1}, j_\nu, \cdots, j_p\}$. 
\end{lemma}

\begin{proof} Let us  first check the case 1) above. Without loss of generality,  we can assume $i_v=i_1=1$ while  $i_2 \ne 1, \cdots, i_p \ne 1$. Now use  the conventions $F_i$ and $k_{i,j}$ to  simplify  notations. Hence the integral becomes
\begin{equation*}
\begin{split}
& \int_{(\Omega\times \mathbb{R}^d)^N } \left(\tilde{f} \cdot k_{1,j_1} \right)  \cdot \left(F_{i_2} \cdot k_{i_2, j_2}\right) \cdots \left(F_{i_p} \cdot k_{i_p, j_p}\right)\bar{f}_N \ud Z\\
=& \int \nabla_{v_1} f(x_1, v_1)\cdot k_{1,j_1}  \left(F_{i_2} \cdot k_{i_2, j_2}\right) \cdots \left(F_{i_p} \cdot k_{i_p, j_p}\right) \ \Pi_{i=2}^N f(x_i, v_i) \ud Z ,
\end{split}
\end{equation*}
where the only term depending on $v_1$ is $f(x_1, v_1)$. Integration by parts shows that (\ref{integral_indicate}) holds. 

In the second case,  without loss of generality, we can assume that $j_1=1$, while $j_2 \ne 1, \cdots,  j_p \ne 1$ and $i_1 \ne 1, \cdots, i_p \ne 1$. Hence the integral becomes 
\begin{equation*}
\begin{split}
&\int F_{i_1}\cdot \{K(x_{i_1}-x_1)-K*\rho(x_{i_1})\} \left(F_{i_2} \cdot k_{i_2, j_2}\right) \cdots \left(F_{i_p} \cdot k_{i_p, j_p}\right)\bar{f}_N \ud Z\\
=& \int_{(\Omega \times \mathbb{R}^d)^{N-1}}\left(F_{i_2} \cdot k_{i_2, j_2}\right) \cdots \left(F_{i_p} \cdot k_{i_p, j_p}\right) \Pi_{i=2}^N f(x_i, v_i)\ud z_2 \cdots \ud z_{N}\\
 & \cdot \int_{\Omega \times \mathbb{R}^d} \nabla_{v_{i_1}} \log f(x_{i_1}, v_{i_1}) \cdot \{K(x_{i_1}-x_1)-K*\rho(x_{i_1})\}f(x_1, v_1) \ud x_1\ud v_1\\
=&\int_{(\Omega \times \mathbb{R}^d)^{N-1}}\left(F_{i_2} \cdot k_{i_2, j_2}\right) \cdots \left(F_{i_p} \cdot k_{i_p, j_p}\right) \Pi_{i=2}^N f(x_i, v_i)\ud z_2 \cdots \ud z_{N}\\
& \cdot \left( \nabla_{v_{i_1}} \log f(x_{i_1}, v_{i_1}) \cdot \int_{\Omega \times \mathbb{R}^d}  \{K(x_{i_1}-x_1)-K*\rho(x_{i_1})\}f(x_1, v_1) \ud x_1\ud v_1 \right),\\
\end{split}
\end{equation*}
where only $K(x_{i_1} -x_1)$ and $f(x_1,v_1)$ are $(x_1, v_1)$-dependent. As in Lemma \ref{vanish2}
\[
\int_{\Omega \times \mathbb{R}^d}  \{K(x_{i_1}-x_1)-K*\rho(x_{i_1})\}f(x_1, v_1) \ud x_1\ud v_1=0,
\]
and hence again (\ref{integral_indicate})  holds, completing the proof. 

\end{proof}

To make easier use of Lemma \ref{generalcancelationrule}, we introduce some definitions, formalizing the set of indices over which the expansion of $R_N$ does not vanish. 
\paragraph{Definitions} In this subsection, we always assume that $3k \leq N$. Recall that we write $I_{p}=(i_1, \cdots, i_{p})$ and $J_{p}=(j_1, \cdots, j_{p})$.  For positive integers $q$ and $p$, 
\begin{itemize}
\item the overall set $\mathcal{T}_{q, p}$ is defined  as 
\[
\mathcal{T}_{q,p}=\{I_p=(i_1, \cdots, i_p)\vert 1\leq i_\nu \leq q, \text{for all\ } 1\leq \nu \leq p \}.
\]
\end{itemize}

Then we define 
\begin{itemize}
\item the multiplicity function $\Phi_{q,p}:\mathcal{T}_{q,p} \to \{0, 1, \cdots, p\}^q$, with $\Phi_{q,p}(I_p)=A_q$,  where $A_q=(a_1, a_2, \cdots, a_q)$ and  $a_l= \vert \{1 \leq \nu \leq p\vert i_\nu=l\}\vert$.
\end{itemize}
With the multiplicity function $\Phi_{q,p}$, we can proceed to define
\begin{itemize}
\item the ``effective set" $\mathcal{E}_{q,p}$ of index $I_p$ as   
\end{itemize} 
\[ 
\mathcal{E}_{q, p}=\{ I_p \in \mathcal{T}_{q,p}\vert\  \Phi_{q,p}(I_p)=A_q=(a_1, \cdots, a_q) \text{ with } a_\nu \ne 1 \text{ for any } 1 \leq \nu \leq q\} .
\]
We can restate case 1) in  Lemma \ref{L3} by using the notation $\mathcal{E}_{N, 2k}$.  That is 
\begin{itemize}
\item  If $I_{2k} \notin \mathcal{E}_{N,2k}$, then 
 \begin{equation}
 \label{fact1}
 \int (F_{i_1}\cdot k_{i_1, j_1}) \cdots (F_{i_{2k}}\cdot k_{i_{2k}, j_{2k}})\bar{f}_N \ud Z=0.
\end{equation}
\end{itemize}
However, even for an  $I_{2k} \in \mathcal{E}_{N,2k}$ , the integral 
\begin{equation*}
\int (F_{i_1}\cdot k_{i_1, j_1}) \cdots (F_{i_{2k}}\cdot k_{i_{2k}, j_{2k}})\bar{f}_N \ud Z
\end{equation*}
can still vanish for some choices of $J_{2k}\text{\ in\ } \mathcal{T}_{N,2k}$ according to the case 2) in Lemma \ref{L3}. Hence, for $I_{2k} \in \mathcal{E}_{N, 2k}$, we  define 
\begin{itemize}
\item the  ``effective set"  $\mathcal{P}_{N,2k}^{I_{2k}}$ of $J_{2k}$ as 
\end{itemize}  
\[
\mathcal{P}_{N,2k}^{I_{2k}}:=\left\{ J_{2k} \in \mathcal{T}_{N, 2k}  \Big\arrowvert \begin{aligned} &  \text{either  for all\ }   1\leq  \nu \leq 2k, j_\nu \in \{i_1, \cdots, i_{2k}\} ; \\ & \text{or for any } \nu  \text{\ such that }j_\nu \notin \{i_1, \cdots, i_{2k}\}, \\&
 \exists \nu' \ne \nu, \text{\ such that\ } j_\nu=j_{\nu'}.  \end{aligned} \right\}
\]
Then  the case 2) in Lemma (\ref{L3}) can be represented as 
\begin{itemize}
\item If $I_{2k} \in \mathcal{E}_{N,2k}$ and $J_{2k} \notin \mathcal{P}_{N,2k}^{I_{2k}}$, then 
\begin{equation}
\label{fact2}
\int (F_{i_1}\cdot k_{i_1, j_1}) \cdots (F_{i_{2k}}\cdot k_{i_{2k}, j_{2k}})\bar{f}_N \ud Z=0.
\end{equation}

\end{itemize}
To simplify the notations in the following proofs, we also define  
\begin{itemize}
\item the set of all components of $I_p =(i_1, i_2, \cdots, i_p)\in \mathcal{T}_{q, p}$ as
\[S(I_p)=\{i_1, i_2, \cdots, i_p \}. \]
\end{itemize}
The set $S(I_p)$ only captures distinct integers in $I_p$. Hence, the cardinality of $S(I_p)$ equals the number of distinct integers  in $I_p$.

\bigskip

We start by bounding $|\mathcal{E}_{q, p}|$
\begin{lemma}
\label{count I}
 Assume that $1 \leq p \leq q$.   Then 
\begin{equation}
\label{boundQp}
 |\mathcal{E}_{q, p}|\leq \sum_{l=1}^{\lfloor \frac{p}{2}\rfloor} \binom{q}{l}l^{p} \leq \lfloor \frac{p}{2}\rfloor \binom{q}{\lfloor \frac{p}{2}\rfloor} \left( \lfloor \frac{p}{2}\rfloor\right)^{p} \leq \frac{p}{2} e^{\frac{p}{2}} q^{\frac{p}{2}} \left( \frac{p}{2}\right)^{\frac{p}{2}}.
\end{equation}  
\end{lemma}

\begin{proof} 
Pick any multi-index $I_p=(i_1, \cdots, i_{p}) \in \mathcal{E}_{q,p}$ and recall  that $S(I_p)=\{i_1, \cdots, i_p\}$. The fact that $I_p \in \mathcal{E}_{q,p}$ implies that the multiplicity of each integer cannot be one. Hence $1 \leq \vert S(I_p)\vert \leq \lfloor \frac{p}{2} \rfloor$.  Indeed, if $|S(I_p)| \geq \lfloor \frac{p}{2} \rfloor +1$, then 
\[ 
p \geq 2 \left(\lfloor \frac{p}{2} \rfloor +1  \right) >2 \left( \frac{p}{2}-1 +1\right)=p,
\]
which is impossible.  

If $p=1$, then $\mathcal{E}_{q,p}=\emptyset$. The estimate (\ref{boundQp}) holds trivially. In the following we assume that $p \geq 2$. We proceed by discussing the cardinality of $S(I_p)$. 

Denote  $|S(I_p)|=l$, where $1 \leq l \leq \lfloor \frac{p}{2} \rfloor$. We count step by step
\[
|\{I_p \in \mathcal{E}_{q, p}\vert |S(I_p)|=l \}|
\]  

\smallskip

{\em Step I:} Choose $l$ distinct integers out of $\{1, 2, \cdots, q\}$. We have $\binom{q}{l}$  choices in this step. Without loss of generality, in the following,  we  assume these $l$ integers are  $1, 2, \cdots, l$, i,e. $S(I_p)=\{1, 2, \cdots, l\}$.

\smallskip

{\em Step II:}  For an $I_p \in \mathcal{E}_{q, p}$ with $S(I_p)=\{1, 2, \cdots, l\}$, recall that the multiplicity function reads
\[
\Phi_{q,p}(I_p)=(a_1, \cdots, a_l, \underbrace{0, \cdots, 0}_{q-l \text{\ times}}).
\]
 For a fixed $l-$tuple $(a_1, a_2, \cdots, a_l)$ with $a_1+a_2+\ldots+a_l=p$, we must choose $I_p=(i_1,\ldots,i_p)$ s.t. all $i_k\in \{1,\ldots,l\}$ and the number $m\in \{1,\ldots,l\}$ is chosen exactly $a_m$ times, which means calculating 
\[
|\{I_p \in \mathcal{E}_{q,p} \vert \  \Phi_{q,p}(I_p)=(a_1, \cdots, a_l,0, \cdots, 0) \}|.
\]
We may choose $a_1$ times the number $1$ among all possible $p$ positions, then $a_2$ times the number $2$ among all remaining $p-a_2$ positions and so on. The total number of choices is  $\frac{p!}{(a_1)! \cdots (a_l)!}$, that is
\begin{equation*}
|\{I_p \in \mathcal{E}_{q,p} \vert \  \Phi_{q,p}(I_p)=(a_1, \cdots, a_l,0, \cdots, 0) \}|=\frac{p!}{(a_1)! \cdots (a_l)!}.
\end{equation*}
 
\smallskip

{\em Step III:} The definition of $\mathcal{E}_{q,p}$ implies that in Step II, 
\[
a_1 + a_2 +\cdots +a_l=p, \text{\  and } a_\nu \geq 2 \text{\ for any } 1 \leq \nu \leq l.
\]
Hence, Step II and Step III gives that 
\[
W_{q,p}^l:=\vert \left\{I_p \in \mathcal{E}_{q, p} \vert \ S(I_p)=\{1, 2,\cdots, l\} \right\}\vert =\sum_{ \substack{ a_1+\cdots + a_l=p,\\a_1 \geq 2,\cdots a_l \geq 2}} \frac{p!}{(a_1)! \cdots (a_l)!},
\]
which is bounded by 
\[
 \sum_{ \substack{a_1+\cdots + a_l=p,\\a_1 \geq 0,\cdots a_l \geq 0}} \frac{p!}{(a_1)! \cdots (a_l)!} = (\underbrace{1+ \cdots + 1 }_{l  \text{ times} })^p=  l^{p}.
\]
Combining all those three steps together, we have that
\[
|\mathcal{E}_{q, p}| = \sum_{l=1}^{\lfloor \frac{p}{2} \rfloor} \vert \{I_p \in \mathcal{E}_{q, p} \vert \ |S(I_p)|=l \}\vert 
= \sum_{l=1}^{\lfloor \frac{p}{2} \rfloor} \binom{q}{l} W_{q,p}^l  \leq \sum_{l=1}^{\lfloor \frac{p}{2} \rfloor} \binom{q}{l} l^p,
\]
where the fact that $p \leq q $ implies  $1 \leq l \leq \lfloor\frac{p}{2}\rfloor \leq \lfloor\frac{q}{2}\rfloor $. Hence,  for $1 \leq l \leq \lfloor\frac{p}{2}\rfloor$, we have  \[
\binom{q}{l} \leq \binom{q}{\lfloor\frac{p}{2}\rfloor},
\] 
leading to, 
\[
|\mathcal{E}_{q, p}|\leq \sum_{l=1}^{\lfloor\frac{p}{2}\rfloor} \binom{q}{l}l^{p}\leq \lfloor\frac{p}{2}\rfloor \binom{q}{\lfloor\frac{p}{2}\rfloor} \left(\lfloor\frac{p}{2}\rfloor \right)^{p}.
\]
The last inequality in (\ref{boundQp}) is now ensured by Stirling's formula. Indeed, write $\lfloor\frac{p}{2}\rfloor=k$,  then Stirling's formula gives 
\[
\binom{q}{k} =\frac{q!}{(q-k)!k!} =\frac{\lambda_q \sqrt{2\pi q} \left(\frac{q}{e}\right)^q}{\lambda_{q-k} \sqrt{2\pi (q-k)} \left(\frac{q-k}{e}\right)^{q-k}\lambda_k \sqrt{2\pi k} \left(\frac{k}{e}\right)^k},
\]
where $\lambda_{q}$, $\lambda_{q-k}$ and $\lambda_{k}$ all lie in $(1, 1.1)$. Hence, 
\begin{equation}
\label{evenbound}
\lfloor\frac{p}{2}\rfloor \binom{q}{\lfloor\frac{p}{2}\rfloor} \left(\lfloor\frac{p}{2}\rfloor \right)^{p} =k\binom{q}{k}k^{2k} \leq \frac{1.1}{\sqrt{2\pi}} \sqrt{k} \sqrt{\frac{q}{q-k}} \left(\frac{q}{q-k}\right)^{q-k} q^k k^k.
\end{equation}
Since 
\begin{equation}
\label{exponetial}
\left(\frac{q}{q-k}\right)^{q-k}= \left( \left(1+ \frac{1}{\frac{q-k}{k}}\right)^{\frac{q-k}{k}}\right)^k \leq e^k,
\end{equation}
and  for any $1 \leq k \leq \frac{q}{2}$, 
\begin{equation*}
\sqrt{\frac{q}{q-k}}  \leq \sqrt{2},
\end{equation*}
we have for $p$ even, 
\begin{equation}
\label{evenQp}
\lfloor\frac{p}{2}\rfloor \binom{q}{\lfloor\frac{p}{2}\rfloor} \left(\lfloor\frac{p}{2}\rfloor \right)^{p} \leq \sqrt{k}e^kq^kk^k =\sqrt{\frac{p}{2}} e^{\frac{p}{2}} q^{\frac{p}{2}} \left( \frac{p}{2}\right)^{\frac{p}{2}} \leq \frac{p}{2} e^{\frac{p}{2}} q^{\frac{p}{2}} \left( \frac{p}{2}\right)^{\frac{p}{2}}.
\end{equation}
For $p$ odd, write instead $p=2k+1 \geq 3$, then by (\ref{evenQp}), we have that
\[\begin{split}
\lfloor\frac{p}{2}\rfloor \binom{q}{\lfloor\frac{p}{2}\rfloor} \left(\lfloor\frac{p}{2}\rfloor \right)^{p} &=k \left( k \binom{q}{k} k^{2k} \right) \leq k \left(\sqrt{k}e^kq^k k^k  \right)=k e^k q^k k^{\frac{p}{2}} \\
&\leq \frac{p}{2} e^{\frac{p}{2}} q^{\frac{p}{2}} \left( \frac{p}{2}\right)^{\frac{p}{2}}. 
\end{split}
\]
This finishes the proof of Lemma \ref{count I}. 
\end{proof}

\bigskip

We now turn to bounding the number of  choices of $J_{2k}$ in $\mathcal{P}_{N, 2k}^{I_{2k}}$ with $I_{2k} \in \mathcal{E}_{N, 2k}$. 
\begin{lemma}
\label{count J} (Choices of the multi-indices $J_{2k}$) Assume that $3k \leq N$ and $I_{2k} \in \mathcal{E}_{N,2k}$ with $|S(I_{2k})|=l$. Recall that $1\leq l \leq k$. Then we have 
\begin{equation}
\label{BJ}
|\mathcal{P}_{N, 2k}^{I_{2k}}|= l^{2k}+ \sum_{h=2}^{2k} l^{2k-h} \binom{2k}{h} |\mathcal{E}_{N-l, h}|.
\end{equation}
Furthermore, 
\begin{equation}
\label{CNK}
|\mathcal{P}_{N, 2k}^{I_{2k}}|  \leq P_{N,2k} := 2k e^k 2^{2k} k^k N^k.
\end{equation}
\end{lemma}
\begin{proof}
Without loss of generality, we assume that $S(I_{2k})=\{1, 2, \cdots, l\}$. By the definition of the set $\mathcal{P}_{N,2k}^{I_{2k}}$, we have two cases. The first case is that  all $j_{\nu}$ lie in the set $S(I_{2k})=\{1, 2, \cdots, l\}$. The total number of such $J_{2k}$ is  $l^{2k}$ since each $j_\nu$ can be any integer from $1$ to $l$. 

In the second case, there exists some  $j_{\nu}$ in $\{l+1, \cdots, N\}$ and for each such $j_\nu \geq l+1$, there exists  $\nu'\ne \nu$ such that $j_\nu = j_{{\nu'}}$. That is to say,   each component  $j_\nu \geq l+1$ is repeated.  Denote by
 \[
h=|\{1 \leq \nu \leq 2k \vert j_\nu \geq l+1\}|
\] 
the number of components of $J_{2k}$ which are larger than $l$.  We thus have $2 \leq h \leq 2k$. 

For a fixed $h$,   we need to choose $h$ positions in $J_{2k}$  to put integers bigger than $l$ for $\binom{2k}{h}$ choices.

The remaining $(2k-h)$ positions of $J_{2k}$ can be filled with any integer in $\{1, 2, \cdots, l\}$,  for $l^{2k-h}$ choices. 

Finally, we choose $h$ integers from the set $\{l+1, \cdots, N\}$ for each of the $h$ positions in $J_{2k}$ that we chose initially. Again, the multiplicity for each integer chosen is at least two and the order is taken into account. This coincides with the definition of $\mathcal{E}_{N-l,h}$.  Hence, in this step, the total number is just $|\mathcal{E}_{N-l, h}|$. 

Therefore for a fixed $h$, one has that 
\[
|\{J_{2k} \in \mathcal{P}_{N, 2k}^{I_{2k}} \vert h \text{\ components of \ } J_{2k} \text{\ are larger than\ } l\}|=\binom{2k}{h} l^{2k-h}|\mathcal{E}_{N-l, h}|.
\]
Adding all the cases together, we obtain
\[\begin{split}
|\mathcal{P}_{N, 2k}^{I_{2k}}|&=l^{2k}+\sum_{h=2}^{2k}|\{J_{2k} \in \mathcal{P}_{N, 2k}^{I_{2k}} \vert h \text{\ components of \ } J_{2k} \text{\ are larger than\ } l\}|\\
&=l^{2k}+\sum_{h=2}^{2k} \binom{2k}{h} l^{2k-h}|\mathcal{E}_{N-l, h}|,
\end{split}\]
which is exactly (\ref{BJ}). 

Now we simplify the bound for $|\mathcal{P}_{N, 2k}^{I_{2k}}|$. 
Applying Lemma \ref{count I}, we have 
\[
|\mathcal{E}_{N-l, h}| \leq \frac{h}{2} e^{\frac{h}{2}}(N-l)^{\frac{h}{2}} \left( \frac{h}{2}\right)^{\frac{h}{2}}. 
\]
Therefore
\begin{equation*}
\begin{split}
 |\mathcal{P}_{N, 2k}^{I_{2k}}|  &\leq   l^{2k}+ \sum_{h=2}^{2k} l^{2k-h} \binom{2k}{h}  h e^{\frac{h}{2}}(N-l)^{\frac{h}{2}} \left( \frac{h}{2}\right)^{\frac{h}{2}} \\ & \leq  l^{2k}+ 2k e^k \sum_{h=2}^{2k} l^{2k-h} \binom{2k}{h} (N-l)^{\frac{h}{2}} k^{\frac{h}{2}} \\
 & \leq  2k e^k \ \left\{ \sum_{h=0}^{2k}  \binom{2k}{h}l^{2k-h} (N-l)^{\frac{h}{2}} k^{\frac{h}{2}}\right\} =2k e^k \ \left( l+\sqrt{k(N-l)}\right)^{2k} \\
 & \leq 2k e^k 2^{2k} k^k N^k. 
\end{split}
\end{equation*}
This completes the proof.

\end{proof}

\bigskip

We are now ready to prove Prop. \ref{ksmall} by  combining Lemma \ref{count I} and Lemma \ref{count J}. 
\begin{proof}[Proof of Prop. \ref{ksmall}]

By the definition of $\mathcal{T}_{q,p}$ for $q=N$ and $p=2k$,   we have by expanding $R_N$ as defined in \eqref{RN}
\[
\begin{split}
&\int |R_N|^{2k} \bar{f}_N \ud Z\\
=&\frac{1}{N^{2k}}\int \sum_{1 \leq i_1, j_1 \leq N} \cdots \sum_{1 \leq i_{2k}, j_{2k} \leq N} (F_{i_1}\cdot k_{i_1, j_1}) \cdots (F_{i_{2k}}\cdot k_{i_{2k}, j_{2k}})\,\bar{f}_N\, \ud Z  \\
=&\frac{1}{N^{2k}} \sum_{I_{2k} \in \mathcal{T}_{N,2k}} \sum_{J_{2k} \in \mathcal{T}_{N,2k}}\int (F_{i_1}\cdot k_{i_1, j_1}) \cdots (F_{i_{2k}}\cdot k_{i_{2k}, j_{2k}})\,\bar{f}_N\, \ud Z.
\end{split}
\]
Applying  Lemma \ref{generalcancelationrule} ( i.e.  facts (\ref{fact1}) and (\ref{fact2})), the previous equality becomes 
\begin{equation}
\label{Summation2}
\int |R_N|^{2k} \bar{f}_N \ud Z=\frac{1}{N^{2k}} \sum_{I_{2k} \in \mathcal{E}_{N,2k}} \sum_{J_{2k} \in \mathcal{P}_{N,2k}^{I_{2k}}}\int (F_{i_1}\cdot k_{i_1, j_1}) \cdots (F_{i_{2k}}\cdot k_{i_{2k}, j_{2k}})\,\bar{f}_N\, \ud Z.
\end{equation} 
For fixed indices $I_{2k} \in \mathcal{E}_{N,2k}$ and $J_{2k} \in \mathcal{P}_{N,2k}^{I_{2k}}$, we have 
\[
\begin{split}
& \int (F_{i_1}\cdot k_{i_1, j_1}) \cdots (F_{i_{2k}}\cdot k_{i_{2k}, j_{2k}})\bar{f}_N \ud Z \\
\leq & \left(2\|K\|_{L^\infty} \right)^{2k}\int \left\vert\nabla_{v_1}f(x_1, v_1) \right\vert^{a_1} \cdots \left\vert\nabla_{v_N}f(x_N, v_N) \right\vert^{a_N} \bar{f}_N \ud Z,
\end{split}
\]
where we recall that $a_\nu$ is the multiplicity of integer $\nu$ in the multi-index $I_{2k}$, i.e. $\Phi_{N,2k}(I_{2k})=A_N=(a_1, \cdots, a_N)$. On the other hand
\[
\begin{split}
&\int \left\vert\nabla_{v_1}f(x_1, v_1) \right\vert^{a_1} \cdots \left\vert\nabla_{v_N}f(x_N, v_N) \right\vert^{a_N}\, \bar{f}_N \ud Z\\ = \quad &M_{a_1}^{a_1} M_{a_2}^{a_2}  \cdots M_{a_N}^{a_N} 
\leq \left( \sup_p \frac{M_p}{p}\right)^{2k} a_1^{a_1}\cdots a_N^{a_N},
\end{split}
\]
with the convention that $0^0=1$. 

Hence, combining (\ref{Summation2}) with the previous inequalities and the second part of Lemma \ref{count J}, we have that
\begin{equation}
\label{PEquation1}
\begin{split}
&\frac{1}{(2k)!}\int |R_N|^{2k}\, \bar{f}_N \\=&\frac{1}{(2k)!}\frac{1}{N^{2k}} \sum_{l=1 }^k  \sum_{I_{2k} \in \mathcal{E}_{N,2k},\  |S(I_{2k})|=l}  \sum_{J_{2k} \in \mathcal{P}_{N,2k}^{I_{2k}}}\int (F_{i_1}\cdot k_{i_1, j_1}) \cdots (F_{i_{2k}}\cdot k_{i_{2k}, j_{2k}})\,\bar{f}_N \\
\leq & \frac{1}{(2k)!}\frac{1}{N^{2k}} \sum_{l=1 }^k \ \sum_{I_{2k} \in \mathcal{E}_{N,2k},\  |S(I_{2k})|=l} P_{N, 2k} \left( 2 \|K\|_{L^\infty} \left( \sup_p \frac{M_p}{p}\right)\right)^{2k} a_1^{a_1}\cdots a_N^{a_N},
\end{split}
\end{equation}
where we recall that $P_{N,2k}=2\,k\,e^k\,2^{2k}\,k^k\,N^k$ which is the bound obtained on $|\mathcal{P}_{N,2k}^{I_{2k}}|$ in Lemma \ref{count J}.
  
Observe that for a given $l$ and given multiplicities $a_1,\ldots, a_l$, the number of $I_{2k} \in \mathcal{E}_{N,2k}$ with such multiplicities is bounded by
\[
\frac{(2k)!}{(a_1)! \cdots (a_l)!}.
\] 
This is the argument in Lemma \ref{count I}, just by choosing first $a_1$ times the number $1$ among all $2k$ positions, then $a_2$ times the number $2$ among the remaining $2k-a_1$ and so on.

Thus
\[
\sum_{l=1 }^k \ \sum_{I_{2k} \in \mathcal{E}_{N,2k},\  |S(I_{2k})|=l} a_1^{a_1}\cdots a_N^{a_N}=\sum_{l=1}^k \binom{N}{l}U_{N,2k}^l,
\]
where 
\begin{equation*}
U_{N, 2k}^l:=\sum_{\begin{aligned}&a_1+\cdots + a_l=2k,\\&a_1 \geq 2,\cdots a_l \geq 2\end{aligned}} \frac{(2k)!}{(a_1)! \cdots (a_l)!} a_1^{a_1}\cdots a_l^{a_l}. 
\end{equation*}  
Combining this estimate with (\ref{PEquation1}), we get 
\begin{equation} 
\label{integral with power $2k$}
 \frac{1}{(2k)!} \int |R_N|^{2k} \bar{f}_N \ud Z 
\leq \frac{(2\|K\|_{L^\infty})^{2k}}{(2k)!} \,\frac{1}{N^{2k}}  \left(\sup_p \frac{M_p}{p} \right)^{2k} P_{N, 2k}\sum_{l=1}^k  \binom{N}{l} U_{N,2k}^l.
\end{equation}
It only remains to simplify the right-hand side of (\ref{integral with power $2k$}).
Since $n^n < e^n n!$,  
\begin{equation}
\label{U1}
U_{N, 2k}^l \leq 
e^{2k}\, (2k)! \,\Big(\sum_{\substack{ a_1+\cdots + a_l=2k,\\a_1 \geq 2,\cdots a_l \geq 2}} 1 \Big)= e^{2k} \,(2k)!\, \binom {2k-l-1}{l-1}.
\end{equation}
Indeed, the equality in (\ref{U1}), i.e.
\[
|\{(a_1, \cdots, a_l)|a_1 \geq 2, \cdots, a_l \geq 2, a_1+\cdots + a_l=2k \}| = \binom{2k-l-1}{l-1}
\]
 comes from the following  classical Combinatorics result where we take  $p=l$ and $b_\nu= a_{\nu}-1$ for $1 \leq \nu \leq p$, with  $q=2k-l$,
\begin{lemma} For integer-valued $p-$tuples $B_p=(b_1, \cdots, b_p) \in \mathcal{T}_{q,p}$, we have
\label{multicombination} 
\[
|\{ B_p \in \mathcal{T}_{q,p}| b_1+\cdots + b_p=q \}| = \binom{q-1}{p-1}. 
\]
\end{lemma}

\begin{proof} [Proof of Lemma \ref{multicombination}]
We give a quick proof for the sake of completeness. Let $c_1=b_1$, $c_2=b_1  + b_2$, $\cdots$, $c_{p-1}=b_1+\cdots + b_{p-1}$. Since $(b_1, \cdots, b_p)$ uniquely determines $(c_1, \cdots, c_{p-1})$ and reciprocally, we only need to check
\[
|\{(c_1, \cdots, c_{p-1}|1 \leq c_1 < c_2 < \cdots <c_{p-1} \leq q-1  \}| =\binom{q-1}{p-1}.
\]
This is simply obtained by choosing any $p-1$ distinct integers from the set $\{1, 2, \cdots, q-1\}$ and assigning the smallest to $c_1$, the second smallest to $c_2$, etc. 
\end{proof}

Coming back to the proof of Prop. \ref{ksmall}, since $1\leq l \leq k$, one has that
\[
  \binom{2k-l-1}{l-1} \leq \binom{2k}{k} \leq \frac{1}{\sqrt{k}} 2^{2k},
\]
by Stirling's formula.  Hence, inserting this bound in  (\ref{U1}),
\[
U_{N, 2k}^l \leq \frac{1}{\sqrt{k}} (2e)^{2k} (2k)!.
\]
Insert into \eqref{integral with power $2k$} this bound for $U_{N, 2k}^l$, and the definition (\ref{CNK}) of $P_{N,2k}$ to obtain that 
\begin{equation}
\label{final step}
\begin{split}
&\frac{1}{(2k)!} \int |R_N|^{2k}\, \bar{f}_N \ud Z \\
&\quad\leq   \frac{(2\|K\|_{L^\infty})^{2k}}{(2k)!}\, \frac{1}{N^{2k}} \, \left(\sup_p \frac{M_p}{p} \right)^{2k}\,  \left(2k\, e^k\, 2^{2k}\, k^k\, N^k \right)\\
&\qquad\qquad\qquad\qquad\qquad\qquad\qquad
 \left( \frac{1}{\sqrt{k}}\, (2e)^{2k}\, (2k)! \right)  \sum_{l=1}^k  \binom{N}{l} 
\\
& \quad\leq 2\, \sqrt{k}\, \left( 8\, \|K\|_{L^\infty}\, \left( \sup_p \frac{M_p}{p}\right)\right)^{2k}\, e^{3k}\, \frac{k^k}{N^k}\, \sum_{l=1}^k \binom{N}{l}\\
& \quad\leq 2\, \sqrt{k}\, \left( 8\, \|K\|_{L^\infty}\, \left( \sup_p \frac{M_p}{p}\right)\right)^{2k}\, e^{3k}\, \frac{k^k}{N^k}\, k\, \binom{N}{k} .
\end{split}
\end{equation}
Now we use Stirling's formula again to simplify the binomial coefficient above,
\[
\frac{k^k}{N^k} \binom{N}{k}=\frac{k^k}{N^k} \frac{N!}{(N-k)! k!}\leq \frac{1}{\sqrt{\pi k}} \sqrt{\frac{N}{N-k}} \left(\frac{N}{N-k}\right)^{N-k} . 
\]
Furthermore, the assumption $3k \leq N$ gives that $\frac{N}{N-k} \leq \frac{3}{2}$.  Thus
\[
\frac{k^k}{N^k} \binom{N}{k} \leq \sqrt{\frac{3}{2\pi k}}  \ e^k.
\]
Using this bound in (\ref{final step}), we get that for $1 \leq k \leq \lfloor \frac{N}{3}\rfloor$, 
\[
\frac{1}{(2k)!} \int |R_N|^{2k} \bar{f}_N \ud Z  \leq  2k \left( 8 e^2 \|K\|_{L^\infty} \left( \sup_p \frac{M_p}{p}\right)\right)^{2k},
\]
finishing the proof of Prop \ref{ksmall}. 
\end{proof}
\subsection{The case $3k > N$: Proof of Proposition \ref{kbig}}  
Now we establish the estimate for large $k$.

\begin{proof}[Proof of Proposition \ref{kbig}] 
We only need the trivial bound for $R_N$, that is 
\begin{equation*}
|R_N| \leq 2 \|K\|_{L^\infty} \sum_{i=1}^N |\nabla_{v_i} \log f|. 
\end{equation*}
Hence, for $k > \frac{N}{3}$, we have
\begin{equation}
\label{klargemain}
\begin{split}
& \frac{1}{(2k)!} \int |R_N|^{2k} \,\bar{f}_N \ud Z
\leq \frac{\left(2\|K\|_{L^\infty}\right)^{2k}}{(2k)!} \int \left( \sum_{i=1}^N |\nabla_{v_i} \log f(x_i, v_i)|\right)^{2k}\, \bar{f}_N \ud Z \\
= &\frac{\left(2\|K\|_{L^\infty}\right)^{2k}}{(2k)!} \sum_{\substack{a_1+\cdots + a_N=2k,\\a_1 \geq 0,\cdots a_N \geq 0}} \frac{(2k)!}{(a_1)! \cdots (a_N)!} M_{a_1}^{a_1} \cdots M_{a_N}^{a_N},\\
\end{split}
\end{equation}
with still the convention that $0!=1=0^0$ and where we recall that
\[
M_{a_i}^{a_i}=\int_{\Omega\times\R^d} |\nabla_{v_i} \log f(x, v)|^{a_i}\,\ud x\,\ud v.
\]
We use again the bound $M_{a_i} \leq a_i\,\sup_p\left( \frac{M_p}{p} \right)$ and hence for $1 \leq i \leq N$,
\[
M_{a_i}^{a_i} \leq a_i^{a_i} \left( \sup_p\frac{M_p}{p} \right)^{a_i} \leq e^{a_i} (a_i)!\left( \sup_p\frac{M_p}{p} \right)^{a_i}.
\]
   Hence, 
\[
M_{a_1}^{a_1} \cdots M_{a_N}^{a_N} \leq  e^{2k} \left( \sup_p\frac{M_p}{p} \right)^{2k}  (a_1)! \cdots (a_N)!.
\]
Therefore the estimate (\ref{klargemain}) becomes 
\begin{equation}
\label{klarge}
\frac{1}{(2k)!} \int |R_N|^{2k} \bar{f}_N \ud Z \leq \left(2 e\|K\|_{L^\infty}\left( \sup_p \frac{M_p}{p} \right)\right)^{2k}  V_{N,2k},
\end{equation}
where $V_{N,2k}=|\{(a_1, \cdots, a_N)\vert a_1 +\cdots +a_N=2k, a_i \geq 0 \text{ for  } 1 \leq i \leq N \}|$.

We can also write $V_{N, 2k} =|\{(b_1, \cdots, b_N)| b_1+\cdots +b_N =2k +N, b_i \geq 1, i=1, \cdots, N\}|$. By Lemma \ref{multicombination},  we have 
\begin{equation*}
V_{N, 2k} = \binom{2k+N-1}{N-1}.
\end{equation*}
We can write $N-1=2k s$, where $s < \frac{3}{2}$, yielding 
\begin{equation*}
\binom{2k+N-1}{N-1}= \frac{(2k(1+s))!}{(2ks)!(2k)!} .
\end{equation*}
Apply Stirling's formula to the factorials above, and notice that $(1+\frac{1}{s})^s < e$ for $s >0$. This shows that for $N \geq 2 $ and $3k > N$,
\[
V_{N, 2k} \leq \binom{2k+N-1}{N-1}\leq  \left(\frac{5}{2}\right)^{2k} e^{2k}.
\]
From this inequality, one obtains that \eqref{klarge} leads to
\begin{equation*}
\frac{1}{(2k)!} \int |R_N|^{2k} \bar{f}_N \ud Z \leq \left(5e^2 \|K\|_{L^\infty} \left( \sup_p \frac{M_p}{p} \right)\right)^{2k}. 
\end{equation*}
Summation over all $k > \frac{N}{3}$ completes the proof. 
\end{proof}
\section{Appendix: Weak-Strong Uniqueness, Proof of Theorem \ref{weakstrongpde} and Prop. \ref{strongexist}}
Let us start with the proof of Theorem \ref{weakstrongpde}.
Assume that $f$ and $\tilde{f}$ solve Vlasov equation  (\ref{PDE}) in weak sense. 
Assume that $f $ satisfies \eqref{explambda}. By density we may assume that $f$ is smooth, $C^1$, and decays at infinity without ever vanishing; just consider any such sequence $f_n$ satisfying uniformly the bound \eqref{explambda} and pass to the limit $f_n\rightarrow f$ at the end of the argument.

Consider for any $t\in [0,\ T]$ and decompose
\[\begin{split}
H(t)&=\int_{\Omega \times \mathbb{R}^d} \tilde{f} \log (\frac{\tilde{f}}{f}) \ud x \ud v=\int \tilde f\,\log \tilde f-\int \tilde f\log f\\
&\leq \int \tilde f^0\,\log \tilde f^0-\eps\int_0^t \int\frac{|\nabla_v \tilde f|^2}{\tilde f}-\int \tilde f\log f,
\end{split}\]
with $\tilde f^0=\tilde f(t=0)$ and per the assumption of dissipation of entropy for $\tilde f$ in Theorem \ref{weakstrongpde}.

By our assumption $f$ is smooth and $\log f$ can hence be used as a test function. Thus since $\tilde f$ is a solution to the Vlasov equation \eqref{PDE} in the sense of distribution, one has that
\[\begin{split}
&\int_{\Omega\times\R^d} \tilde f\, \log f=\int_{\Omega\times\R^d} \tilde f^0\,\log f^0\\
&+\int_0^t\!\int_{\Omega\times\R^d} \tilde f(s,x,v)\,(\partial_t \log f+v\cdot\nabla_x \log f+K\star\tilde \rho\cdot\nabla_v \log f+\eps\Delta_v \log f). 
\end{split}
\]
Since $f$ is a strong solution to the Vlasov equation, this leads to 
\[\begin{split}
&\int \tilde f\, \log f=\int \tilde f^0\,\log f^0+\int_0^t\!\int \tilde f(s,x,v)\,R\ud x \ud v \ud s\\
&+\eps\,\int_0^t\!\int \tilde f(s,x,v)\,\left(\frac{\Delta_v f}{f}+ \Delta_v \log f\right) \ud x \ud v \ud s,
\end{split}
\]
where we define
\[
R:=\nabla_v \log f(x, v)  \cdot \{ K\star \tilde{\rho}(x) -K\star\rho(x)\}.
\]
Observe now that, with usual entropy estimates
\[\begin{split}
&-\int \tilde f(s,x,v)\,\left(\frac{\Delta_v f}{f}+ \Delta_v \log f\right) \ud x\ud v-\int \frac{|\nabla_v \tilde f|^2}{\tilde{f}} \ud x \ud v\\
&\qquad=\int \left(-\tilde f\,\frac{|\nabla_v f|^2}{f^2}+2\frac{\nabla_v \tilde f\cdot\nabla_v f}{f}-\frac{|\nabla_v \tilde f|^2}{ \tilde{f}}\right)\ud x \ud v\\
&\qquad=-\int \tilde f\,|\nabla_v \log \frac{f}{\tilde f}|^2 \ud x \ud v\leq 0.
\end{split}
\]
Therefore
\begin{equation}
\label{ERE}
H(t)\leq H(0) -\int_0^t\int_{\Omega \times \mathbb{R}^d} \tilde{f} R \ud x \ud v\ud s.
\end{equation}
Note that by the definition of $R$
\[
\int_{\Omega \times \mathbb{R}^d} f\, R \,dx\,dv=\int \nabla_v f\,(K\star \tilde \rho-K\star \rho) \ud x \ud v=0,
\]
as $K\star \rho$ and $K\star \tilde \rho$ do not depend on $v$. Hence
\[
\int_{\Omega \times \mathbb{R}^d} \tilde{f} R \ud x \ud v=\int_{\Omega \times \mathbb{R}^d} (\tilde{f}-f)\, R \ud x \ud v.
\]
Simply bound
\[
\left\vert \int_{\Omega \times \mathbb{R}^d} \tilde{f} R \ud x \ud v \right\vert \leq \|K\star (\tilde \rho-\rho)\|_{L^\infty}  \int_{\Omega \times \mathbb{R}^d} |\nabla_v \log f| \ |\tilde{f}-f| \ud x \ud v.
\]
Observe that
\[
\|K*(\tilde{\rho} -\rho)\|_{L^\infty} \leq \|K\|_{L^\infty} \|\tilde{\rho}-\rho\|_{L^1} \leq \|K\|_{L^\infty} \|\tilde{f}-f\|_{L^1},
\]
so that
\[
H(t) \leq H(0)+\|K\|_{L^\infty} \, \|\tilde{f} -f\|_{L^1}\, \int_0^t \int_{\Omega \times \mathbb{R}^d} |\nabla_v \log f| \ |\tilde{f}-f| \ud x\ud v \ud s .
\]
Use the weighted CKP inequality in Theorem  1 in \cite{BV} with $\varphi(x, v) =|\nabla_v \log f|$ to obtain
\[
\int |\nabla_v \log f| |\tilde{f}-f| \ud x \ud v \leq \frac{2}{\lambda}
\left( \frac{3}{2} + \log \int e^{ \lambda|\nabla_v \log f|} f \ud x \ud v \right) \left(\sqrt{H} + \frac{1}{2} H\right).
\]
Recall the notation 
\[
\theta_f=\sup_{t\in [0,\ T]}  \int e^{\lambda |\nabla_v \log f|}\, f\, \ud x \ud v<\infty,
\]  
by the assumption \eqref{explambda}. This leads to 
\[
{H}(t) \leq H(0)+C\,(1+\log \theta_f)\, \|K\|_{L^\infty}\,\|f-\tilde f\|_{L^1}\, \int_0^t \left(\sqrt{H}+\frac{H}{2}\right) \ud s.
\]
Simply use now the classical CKP inequality (see \cite{V}) to find
\begin{equation}
\label{gronwallweakstrong}
{H}(t) \leq H(0)+C\,(1+\log \theta_f)\, \|K\|_{L^\infty}\, \int_0^t \left(H+\frac{H^{3/2}}{2}\right) \ud s.
\end{equation}
As long as $H(t) \leq 1$, then $H^{\frac{3}{2}} \leq H$. Eq. \eqref{gronwallweakstrong} gives a Gronwall's inequality which proves Theorem \ref{weakstrongpde}.

\bigskip

\noindent {\bf Proof of Prop. \ref{strongexist}}. We first denote the linear operator for a fixed $\rho(t,x)$ as 
\[
L= v \cdot \nabla_x f + K \star \rho \cdot \nabla_v .
\]
To show the existence of a smooth solution over a short time, it is sufficient to propagates some norms of $|\nabla f|$. 

\medskip

\noindent \textbf{Step I}: Propagate $\|\nabla f \|_{L^1}$ and $\|\nabla f\|_{L^\infty}$. 
It is easy to check that 
\begin{equation}
\label{VlasovDeri}
\begin{cases}
\partial_t (\nabla_x f ) + L(\nabla_x f) =\eps \Delta_v (\nabla_x f) - (K \star \nabla_x \rho ) \cdot \nabla_v f ,\\
\partial_t (\nabla_v f ) + L(\nabla_v f) =\eps \Delta_v (\nabla_v f) - \nabla_x f .
\end{cases}
\end{equation}
In the following, we also write 
\[
\nabla f = \left(\begin{array}{c}
\nabla_x f \\
\nabla_v f \\
\end{array} \right).
\]
Hence the equation (\ref{VlasovDeri}) can be written as 
\[
\partial_t (\nabla f) + L (\nabla f) =\eps \Delta_v(\nabla f) -\left( \begin{array}{c}
(K \star \nabla_x \rho ) \cdot \nabla_v f\\
\nabla_x f \\
\end{array}\right).
\]
The evolution of  $\|\nabla f\|_{L^1}$ is given by 
\[\begin{split}
\frac{\ud }{\ud t} \|\nabla f\|_{L^1} &\leq \left( \|K \star \nabla_x \rho\|_{L^\infty} + 1 \right) \|\nabla f \|_{L^1}\left(\|K\|_{L^\infty} \|\nabla \rho\|_{L^1} +1 \right) \|\nabla f\|_{L^1}\\
 &\leq \left(\|K\|_{L^\infty} \|\nabla f\|_{L^1} +1 \right) \|\nabla f\|_{L^1}.
\end{split}\]
This is a closed inequality as the right-hand side only depends on $\|\nabla f\|_{L^1}$. This may blow-up in finite time because of the $\|\nabla f\|_{L^1}^2$. However there exists $T>0$ which depends only on $\|\nabla f^0\|_{L^1}$ s.t. $\sup_{t\leq T} \|\nabla f\|_{L^1}<\infty$. This is the time interval over which Prop. \ref{strongexist} holds.

By the maximum principle, we can now bound $\| \nabla f\|_{L^\infty}$ up to this time $T$. Indeed
\[
\frac{\ud }{\ud t} \|\nabla f\|_{L^\infty}  \leq \left(\|K\|_{L^\infty} \|\nabla f\|_{L^1} +1 \right) \|\nabla f\|_{L^\infty} \lesssim \|\nabla f\|_{L^\infty}. 
\]
Observe that there cannot be any blow-up in $\|\nabla\|_{L^\infty}$ before there is blow-up in $\|\nabla\|_{L^1}$.

To conclude this step, we have obtained a time $T>0$, s.t. 
\[
\|\nabla f \|_{L^1} \leq C,  \quad   \|\nabla f\|_{L^\infty} \leq C,\quad \forall t\leq T,
\]
where $C$ depends on $\|K\|_{L^\infty}$,  $\|\nabla f^0\|_{L^1}$ and $\|\nabla f^0\|_{L^\infty}$.

\medskip

\noindent\textbf{Step II}: Define the variable quantity 
\[
\Theta_f(t, \lambda) := \int_{\Omega \times \mathbb{R}^d} f \exp(\lambda|\nabla \log f |) \ud x \ud v .
\]
The main object below is to bound $\Theta_f(t, \lambda)$ in $[0, T]$ for some $\lambda$ as the estimate required for weak-strong uniqueness argument is  
\[
\sup_{t \in [0, T]}\int f \exp(\lambda |\nabla_v \log f |)\ud z < \infty.
\]
First, we derive the equation for $\exp(\lambda|\nabla \log f|)$. Denote
\[
\vec{N}=\nabla \log f =\left( \begin{array}{c} \vec{N}_x\\ \vec{N}_v\end{array}\right) =\left( \begin{array}{c}\nabla_x \log f\\ \nabla_v \log f \end{array}\right),  \quad 
\vec{n}= \frac{\vec{N}}{|\vec{N}|}.
\]
By Eq. \eqref{VlasovDeri}, one has that
\[
\begin{split}
&(\partial_t+ L ) \exp(\lambda |\nabla \log f|) =\lambda \exp(\lambda |\nabla \log f|) \vec{n} \cdot (\partial_t +L)\vec{N}\\
&= \lambda \exp(\lambda |\nabla \log f|) \vec{n} \cdot  \left( \begin{array}{c}
\!\!\!-(K \star \nabla_x \rho ) \cdot \nabla_v \log f + \frac{\eps}{f}\,(\Delta_v (\nabla_x f) -\nabla_x \log f \Delta_v f )\!\!\!\\
-\nabla_x \log f +\frac{\eps}{f}\, (\Delta_v(\nabla_v f) -\nabla_v \log f \Delta_v f )
\end{array} \right)\\
&\leq  C \lambda  \exp(\lambda |\nabla \log f|) |\nabla \log f |\\
&\hspace{80pt} + \eps \lambda \frac{1}{f} \exp(\lambda |\nabla \log f|) \vec{n} \cdot  \left( \begin{array}{c} \Delta_v (\nabla_x f) -\nabla_x \log f \Delta_v f \\ \Delta_v(\nabla_v f) -\nabla_v \log f \Delta_v f \end{array} \right).\\
\end{split}
\]
Thus 
\[
\begin{split}
&\partial_t (f \exp(\lambda |\nabla \log f|)) +L (f \exp(\lambda|\nabla \log f |))\\
&\leq   C \lambda f \exp(\lambda |\nabla \log f|) |\nabla \log f | + \eps \exp(\lambda|\nabla \log f |) \Delta_v f\\
&\hspace{50pt} + \eps \lambda  \exp(\lambda|\nabla \log f |) \vec{n} \cdot \left( \begin{array}{c}
\Delta_v (\nabla_x f) -\nabla_x \log f \Delta_v f \\
\Delta_v(\nabla_v f) -\nabla_v \log f \Delta_v f \end{array} \right).\\
\end{split}
\]
Hence, by integration by parts,
\[
\frac{\ud }{\ud t} \int_{\Omega \times \mathbb{R}^d} f \exp(\lambda |\nabla \log f|) \ud z  \leq C\lambda \int f \exp(\lambda |\nabla \log f |) |\nabla \log f| + Q_\epsilon,
\]
where $Q_\epsilon$ is an extra term due to the diffusion,
\[
\begin{split}
Q_\epsilon = &\epsilon \lambda \int \frac{\exp(\lambda|\nabla \log f|)}{|\nabla \log f |} \Big( \nabla_x \log f \cdot \Delta_x (\nabla_x f) - |\nabla_x \log f|^2 \Delta_v f +\\
& \qquad \nabla_v \log f \cdot \Delta_v(\nabla_v f)- |\nabla_v \log f|^2 \Delta_v f\Big) + \eps \int \exp(\lambda |\nabla \log f|) \Delta_v f.
\end{split}
\]
Notice that 
\[\begin{split}
(\nabla_x  \log f)\cdot \Delta_v(\nabla_x f) =&|\nabla_x \log f|^2 \Delta_v f + 2(\nabla_x \log f) \cdot (\nabla_v f \cdot \nabla_v ) (\nabla_x \log f)\\
& + f \nabla_x \log f \cdot \Delta_v (\nabla_x \log f) ,
\end{split}\]
and 
\[
\begin{split}
(\nabla_v  \log f)\cdot  \Delta_v(\nabla_v f) =&|\nabla_v \log f|^2 \Delta_v f + 2(\nabla_v \log f) \cdot (\nabla_v f \cdot \nabla_v ) (\nabla_v \log f)\\
& + f \nabla_v \log f \cdot \Delta_v (\nabla_v \log f).
\end{split}
\]
We hence obtain that
\[
\begin{split}
Q_\eps &= 2 \lambda \eps \int f\exp(\lambda|\nabla \log f|) \vec{n} \cdot( \vec{N}_v \cdot \nabla_v)\vec{N} + \lambda \eps \int f\exp(\lambda|\nabla \log f|) \vec{n} \cdot \Delta_v \vec{N}\\
&\qquad + \eps \int \exp(\lambda |\nabla \log f|) \Delta_v f \\
&=\lambda \eps \int f\exp(\lambda|\nabla \log f|)  \ \vec{N}_v(\nabla_v \vec{N}\vec{n})+ \lambda \eps \int f\exp(\lambda|\nabla \log f|) \vec{n} \cdot \Delta_v \vec{N}\\
&= \lambda \eps \int f\exp(\lambda|\nabla \log f|)  \ \vec{N}_v(\nabla_v \vec{N}\vec{n})- \lambda \epsilon \sum_{i=1}^{2d} \int f\exp(\lambda|\nabla \log f|)\nabla_v N_i \cdot \nabla_v n_i\\
&\qquad-\lambda \eps \sum_{i=1}^{2d} \int f\exp(\lambda|\nabla \log f|)(\vec{N}_v + \lambda \nabla_v \vec{N}\vec{n})n_i \nabla_v N_i \\
&\qquad\\
&= -\lambda^2 \epsilon \int f\exp(\lambda|\nabla \log f|) |\nabla_v \vec{N} \vec{n}|^2 -\lambda \eps \int f\exp(\lambda|\nabla \log f|) \nabla_v \vec{N} \cdot \nabla_v \vec{n}\\
&\leq 0.
\end{split}
\]
Hence, 
\[
\frac{\ud }{\ud t} \int_{\Omega \times \mathbb{R}^d} f \exp(\lambda |\nabla \log f|) \ud z  \leq C\lambda \int f \exp(\lambda |\nabla \log f |) |\nabla \log f|.
\]
That is 
\[
\partial_t \Theta_f-C\lambda \partial_\lambda \Theta_f \leq 0.
\]
The characteristic equation is given by $\lambda(t) = \lambda_0 e^{-Ct}$ which implies 
\[
\Theta_f(t, \lambda(t)) \leq \Theta_f(0, \lambda_0) =\int f \exp(\lambda_0 |\nabla \log f |) < \infty.
  \] 
  Hence we get 
\begin{equation*}
\int f \exp(\lambda_0 e^{-Ct} |\nabla \log f |) \leq \Theta_f(0)< \infty.
\end{equation*}
Consequently \eqref{explambda} holds for $\lambda < \lambda_0 e^{-CT}$, where $C=\|K \star  \nabla_x \rho\|_{L^\infty}+1 <\infty$.  \\

In the case $\eps=0$, we can easily propagate the bound for $|\nabla \log f |$ by tracing back the characteristics.


\end{document}